\documentclass[11pt]{amsart}
\usepackage[dvips]{epsfig}
\usepackage{graphics}
\usepackage{latexsym}
\usepackage{verbatim}
\usepackage{amsmath}
\usepackage{amsthm}
\usepackage{amssymb}
\usepackage{hyperref}
\usepackage{euscript}
\usepackage{mathrsfs}
\usepackage[all]{xy}
\usepackage{amscd}
\usepackage{graphicx}

\newtheorem{theorem}{Theorem}[section]
\newtheorem{lemma}[theorem]{Lemma}
\newtheorem{corollary}[theorem]{Corollary}
\newtheorem{proposition}[theorem]{Proposition}

\newtheorem{remark}[theorem]{Remark}

\newcommand{\bi}{\begin{itemize}}
\newcommand{\ei}{\end{itemize}}

\newcommand{\be}{\begin{equation}}
\newcommand{\ee}{\end{equation}}
\newcommand{\bdm}{\begin{displaymath}}
\newcommand{\edm}{\end{displaymath}}
\newcommand{\ba}{\begin{array}}
\newcommand{\ea}{\end{array}}

\newcommand{\ud}{\textrm{d}}
\newcommand{\dt}{\partial _{t}}

\newcommand{\dx}{\partial _{x}}

\newcommand{\dz}{\partial _{z}}

\newcommand{\R}{\mathbb{R}}
\newcommand{\Z}{\mathbb{Z}}
\newcommand{\N}{\mathbb{N}}
\newcommand{\C}{\mathcal{C}}

\newcommand{\PP}{\mathcal{P}}

\newcommand{\Ecal}{\mathcal{E}}
\newcommand{\Rcal}{\mathcal{R}}

\newcommand{\dv}{\partial _{v}}
\newcommand{\du}{\partial _{u}}
\newcommand{\dw}{\partial _{w}}

\newcommand{\eps}{\varepsilon}
\newcommand{\ueps}{u^{\varepsilon}}
\newcommand{\veps}{v^{\varepsilon}}

\def\signff{\bigskip\bigskip\hspace{80mm}
\vbox{{\sc Francis Filbet\par\vspace{3mm}
Universit\'e de Lyon,\par
UL1, INSAL, ECL, CNRS \par 
UMR5208, Institut Camille Jordan,\par
43 boulevard 11 novembre 1918,\par
F-69622 Villeurbanne cedex,  FRANCE
\par\vspace{3mm}e-mail:} filbet@math.univ-lyon1.fr }}

\def\signar{\bigskip\bigskip\hspace{80mm}
\vbox{{\sc Am\'elie Rambaud\par\vspace{3mm}
Universit\'e de Lyon,\par
UL1, INSAL, ECL, CNRS \par 
UMR5208, Institut Camille Jordan,\par
43 boulevard 11 novembre 1918,\par
F-69622 Villeurbanne cedex,  FRANCE
\par\vspace{3mm}e-mail:} rambaud@math.univ-lyon1.fr }}

\begin{document}

\title[An  Asymptotic Preserving Scheme for Relaxation Systems]
{Analysis of an Asymptotic Preserving Scheme for Relaxation Systems}\thanks{F. Filbet is partially supported by the European Research Council ERC Starting Grant 2009,  project 239983-\textit{NuSiKiMo}}

\author{Francis Filbet and Am\'elie Rambaud}

\hyphenation{bounda-ry rea-so-na-ble be-ha-vior pro-per-ties
cha-rac-te-ris-tic}

\maketitle

\begin{abstract} 
We study the convergence of a class of asymptotic preserving  numerical schemes initially proposed by F. Filbet \& S. Jin \cite{filb1} and G. Dimarco \& L. Pareschi \cite{DimarcoP} in the context of nonlinear and stiff kinetic equations. Here, our analysis is devoted to the approximation of a system of transport equations with a nonlinear source term, for which the asymptotic limit is given by a conservation laws. We investigate the convergence of the approximate solution $(\ueps_h,\veps_h)$ to a nonlinear relaxation system, where  $\eps>0$ is a physical parameter and $h$ represents the discretization parameter. Uniform convergence with respect to $\eps$ and $h$ is proved and error estimates are also obtained. Finally, several numerical tests are performed to illustrate the accuracy and efficiency of such a scheme. 
\end{abstract}

\tableofcontents

\section{\bf Introduction}
\label{sec:1}
\setcounter{equation}{0}

Many physical problems are governed by hyperbolic conservation laws with non vanishing stiff source terms. These problems can describe the effect of relaxation toward an equilibrium, as in the kinetic theory of gases for example. Actually for monatomic gases, when an equilibrium state is perturbed, it gradually relaxes to the equilibrium state with Maxwellian velocity distribution. 
The time scale for such a relaxation process may not be short and the phenomenon of thermo-nonequilibrium becomes important. In this case, the compressible Euler or Navier-Stokes equations can be applied. In the domain of kinetic equations, the relaxation parameter is represented by the dimensionless Knudsen number $\eps>0$, defined as the ratio of the mean free path of the particles over a typical length scale, such as the size of the spatial domain. At the continuous level, it has been shown \cite{asympKin2} that, for the Boltzmann equation
$$
\displaystyle{\dt f + v\cdot \nabla_x f = \frac{1}{\eps}\,Q(f)\,,} 
$$
when the Knudsen number $\eps$ goes to zero, the distribution function $f$ converges to a local Maxwellian $\mathcal{M}$,  so the macroscopic model (compressible Euler or Navier-Stokes) becomes more adequate to describe the behavior. For more details about fluid dynamic limits of kinetic equation, we refer to \cite{ asympKin2,asympKin3, asympKin4,asympKin1}.  

In this context, developing robust numerical schemes for kinetic equations that also work in the fluid regime becomes challenging. It has been done in the framework of \emph{Asymptotic-Preserving} (AP) schemes \cite{filb1, jinAP}. As it has already been defined in \cite{jinAP}, a numerical scheme is Asymptotic Preserving if
\begin{enumerate}
 \item it is a suitable scheme for the relaxation problem (the kinetic equation), and when the mesh size and time step are fixed, whereas the parameter $\eps$ goes to zero, the scheme becomes a suitable discretization of the limit (equilibrium) problem;
\item implicit collision terms can be implemented explicitly.
\end{enumerate}

Indeed, from a numerical point of view, the treatment of the stiffness can not be done with explicit schemes, so mathematicians will rather favor the use of semi-implicit of fully implicit schemes.

One solution offered by E. Gabetta, L. Pareschi and G. Toscani \cite{gabetta} to design an Asymptotic Preserving scheme, was to penalize the nonlinear collision operator $Q(f)$ by a linear function $\lambda \, f$ and then absorb the linearly stiff part into the time variable to remove the stiffness.  More recently, F. Filbet \& S. Jin \cite{filb1} proposed to penalize the Boltzmann's operator by the BGK operator in order to build stable schemes with respect to $\eps>0$. If such schemes are now numerically validated and extensively used to discretize kinetic equations  \cite{ DimarcoP,filb1, filb2,gabetta,jinAP}, their mathematical study  has only been done in some particular cases \cite{golse_jin_lev}, because of the complexity of general kinetic equations, where the collision operator is nonlinear. 
In this context,  hyperbolic conservation laws represent a simplified framework where a theoretical study of relaxation schemes can be done \cite{ar_nat, chalabi1, chalabi3, chalabi2, ajout:1,filb1, jinAP, jinRK, ajout:0}. Indeed, there is a strong analogy between the local relaxation approximation of conservation laws, initially proposed in \cite{jinAP}, and the study of fluid dynamical limits of kinetic equations \cite{asympKin1}. In the domain of hyperbolic conservation laws with stiff source terms, the relaxation parameter plays the role of the Knudsen number in kinetic theory, while the source term is the analogous of the collision operator, which we want to be the most general possible.  
Few works are devoted to the mathematical analysis of relaxation scheme for the approximation of conservations laws. We refer for instance to the series of paper by D. Aregba-Driollet and R. Natalini \cite{ar_nat} and A. Chalabi \cite{ chalabi1, chalabi3, chalabi2}. But for all of them, the relaxation operator is relatively simple and can be easily treated explicitly without any additional computational cost. Moreover, the goal of these works mainly consists to approximate the asymptotic limit of the relaxation model and not to develop a scheme which is accurate and efficient in various regimes.

Here we want to focus on  the approximation of general relaxation system with a nonlinear and stiff source term as in the context of the Boltzmann kinetic equation, which is valid both for rarefied (large values of $\eps$) and hydrodynamic ($\eps\rightarrow 0$) regimes. Therefore, we will consider throughout the rest of the paper the following hyperbolic relaxation system. For all $(t,x)$ in $\R^+ \times \R$:
\begin{equation}
\label{our:sys1}
\left\{ 
\begin{array}{l}
\displaystyle{\dt \ueps + \dx \veps = 0 \,,}
\\
\,
\\
\displaystyle{\dt \veps + a\, \dx \ueps = -\frac{1}{\varepsilon} \,\mathcal{R}(\ueps,\veps),}
\end{array}
\right.
\end{equation}
where $ a \, > \,0$ is a constant coefficient to be discussed later, $\varepsilon$ is the relaxation parameter and $\Rcal : \, \R \times \R \, \mapsto \, \R $ is a nonlinear function. The system is completed with the initial conditions:
\begin{equation}
\label{our:sys2}
 \left\{ 
\begin{array}{l}
 \displaystyle{\ueps (0,x)  = \ueps_0 (x) \,,}
\\
\,
\\
\displaystyle{\veps (0,x) = \veps_0 (x) \,.}
\end{array}
\right.
\end{equation}

The system of equations (\ref{our:sys1})-(\ref{our:sys2}) is often referred as a two velocity kinetic equation. Here we will assume that the function $\Rcal\in\C^1(\R\times\R, \R)$  possesses a unique local equilibrium, restricted to the manifold $\displaystyle{\{ v = A(u)\}}$, that is,
 \begin{equation}
\label{hyp:00}
\displaystyle{\Rcal (u,v) = 0 \, \Leftrightarrow \, v = A(u)\,,}
\end{equation}
where $A$ is a locally Lipschitz continuous function with $\displaystyle{A(0) = 0}$. Therefore, under some assumptions on $\Rcal$ and on the initial data, the solution $(\ueps,\veps)$ to (\ref{our:sys1})-(\ref{our:sys2}) converges to $(u,v)$ with $v=A(u)$ and $u$ solution to the conservation laws \cite{chen,liu}
\begin{equation}
\label{equi:sys1}
\left\{ 
\begin{array}{l}
\displaystyle{\dt u + \dx A(u) = 0 \,,\textrm{ in }\R^+ \times \R,}
\\
\,
\\
\displaystyle{u(t=0) = u_0\,,}
\end{array}
\right.
\end{equation}
where the initial datum $u_0$ is given by 
\begin{equation}
\label{defu0}
 u_0 \, = \, \underset{\eps \rightarrow 0}{\lim} u_0^{\eps} \,.
\end{equation}

Concerning the mathematical study of the system (\ref{our:sys1})-(\ref{our:sys2}), we refer for instance  to \cite{nat}.
\begin{theorem}
\label{theo_contin:1}
Assume that the initial datum $(\ueps_0,\veps_0)$ is bounded independently of $\eps$ in $L^{\infty} \cap BV(\R)$. Consider $\Rcal\in\C^1(\R\times\R, \R)$, which satisfies (\ref{hyp:00}) and  $\nabla\Rcal$ is uniformly bounded with respect to $v$ and locally bounded with respect to $u$ such that, for any $(u,v)$ in $[-U_0,U_0]\times \R$:
\begin{equation}
\label{hyp:01}
\left\{
\begin{array}{l}
\displaystyle{ \left|\du \Rcal (u,v)\right| \, \leq \, g(U_0)\,,}
\\
\;
\\
\displaystyle{  0\,< \, \beta_0(U_0) \, \leq \, \dv \Rcal (u,v) \, \leq \, h(U_0) \,,}
\end{array}\right.
\end{equation}
where $\beta_0$, $g$ and $h$ are some constants depending only  on $U_0$. Then there exists a characteristic speed $a>0$ large enough (see the condition (\ref{hyp:03}) below) such that the system (\ref{our:sys1})-(\ref{our:sys2}) admits a unique solution $(\ueps, \veps)$ in $\mathcal{C}\left([0, \infty[, L^{1}_{loc} (\R)^2\right)$ and there exists a constant $C>0$ which only  depends on $a$ and $(\ueps_0,\veps_0)$, such that for any $\eps > 0$:
\begin{equation}
\label{resu:01}
\left\{
\begin{array}{l}
\displaystyle{\left\| \veps(t) \pm \sqrt{a}\, \ueps(t) \right\|_{{\infty}} \, \leq \,C \quad \forall \, t > 0,}
\\
\,
\\
\displaystyle{TV\left(\ueps(t)\right) \,+\,  TV\left(\veps (t)\right)  \,\, \leq\,\,C \quad \forall \, t > 0,}
\\
\,
\\
\displaystyle{\|\ueps(t+\tau) -\ueps(t)\|_{1} \,\,\leq\,\, C\,\tau, \quad \forall t\in\R^+,\,\tau\in\R^+,}
\\
\, 
\\
\displaystyle{\|\veps(t+\tau) -\veps(t)\|_{1} \,\,\leq\,\, \frac{C}{\eps}\,\tau, \quad \forall t\in\R^+,\,\tau\in\R^+,}
\\
\,
\\
\displaystyle{\|\veps(t+\tau) -\veps(t)\|_{1} \,\,\leq\,\, C_\nu \,\tau, \quad \forall t \, \geq \, \nu,\,\tau\in\R^+,}
\end{array}\right.
\end{equation}
where $\nu > 0$, and $C_\nu$ only depends on $a$, $(\ueps_0,\veps_0)$ and $\nu$.
Moreover, there exists $\beta_0>0$ such that,
\begin{equation}
\label{resu:02}
  \left\|\veps(t,\cdot)-A\left(\ueps(t,\cdot)\right)\right\|_{1} \,\,\leq \,\,e^{-\frac{\beta_0 t}{\eps}} \, \left\|v_0^{\eps}-A\left(u_0^{\eps}\right)\right\|_{1} \,+\, C\,\eps.
\end{equation}
Finally, if $\ueps_0$ converges, as $\eps$ goes to zero, to $u_0$ defined by  (\ref{defu0}), then  the sequence $(\ueps, \veps)$  converges to $(u,A(u))$ when $\eps$ goes to $0$, such that, for any $\nu > 0$: 
\begin{equation}
\left\{
\begin{array}{l}
\ueps \, \longrightarrow \, u \quad \mathcal{C}\left([0, \infty); L^1_{loc}(\R)\right),
\\
\,
\\
 \veps \, \longrightarrow \, A(u) \quad \mathcal{C}\left([\nu, \infty); L^1_{loc}(\R)\right),
\end{array}
\right.
\end{equation}
where $u$ is the unique entropic solution to the Cauchy problem (\ref{equi:sys1}).
\end{theorem}

In this paper we propose a rigorous analysis of the Asymptotic Preserving scheme proposed by F. Filbet \& S. Jin \cite{filb1} and G. Dimarco \& L. Pareschi \cite{DimarcoP} for a nonlinear relaxation. In other words, denoting by $\PP^{\eps}$  and $\PP^0$ respectively the relaxation and the equilibrium Cauchy problems, and $\PP_h^{\eps}$ and $\PP_h^0$ the corresponding discrete problems, where $h$ represents the discretization parameter, independent of $\eps$, we will perform a precise analysis of the following Asymptotic Preserving diagram.  

$$
\xymatrix{
{ \PP^{\eps}_h} \ar[dd]_{\begin{array}{l}h  \\\downarrow \\ 0\end{array}} \ar[rrrr]^{ \eps \, \rightarrow \, 0}    &\,&\, &\,& { \PP^0_h} \ar[dd]^{\begin{array}{l}h  \\\downarrow \\ 0\end{array}}   
\\
\, &\,&\, &\, &\,
\\
{ \PP^{\eps}} \ar[rrrr]_{ \eps \, \rightarrow \, 0}   &\,&\,&\,& { \PP^0}   } 
$$

The paper is organized as follows. We present in Section \ref{sec:discretizations} an asymptotic preserving scheme for the  relaxation model, and state the both convergence results of the Asymptotic Preserving scheme when the relaxation parameter $\eps$ goes to zero (Theorem~\ref{theo:00}) and next when the discretization parameter $h$ goes to zero (Theorem~\ref{theo:01}). Then, we establish some \emph{a priori} estimates in $L^{\infty}$ and $BV$ on the numerical solution to the Asymptotic Preserving scheme in Section \ref{sec:apriori} in order to prove both the zero relaxation limit (Section~\ref{sec:deviation}) and the convergence of the scheme (Section~\ref{sec:convergence}). Notice that these {\it a priori} estimates are uniform both with respect to the relaxation parameter $\eps$ and the time step $\Delta t$ and space step $\Delta x$.  Finally, we present some numerical results in Section \ref{sec:simul}. 

\section{\bf Numerical schemes and main results}
\label{sec:discretizations}
\setcounter{equation}{0}

We remind that when $\Rcal(u,v) = v - A(u)$, where $A\in\C^1(\R,\R)$ is a given function,  the necessary and sufficient stability condition is given by the so called sub-characteristic condition \cite{BianHanNat,jinAP}:
\begin{equation}
\label{sub:classic}
|A'(u)| \, < \, \sqrt{a}\,.
\end{equation}
It means that the propagation speed of the equilibrium problem has to be bounded by the speeds of the relaxation system, which has to be dissipative. For more details about this case, we refer to \cite{chen,liu,nat}. 

Hence the sub-characteristic condition reads, in our case:
$$\displaystyle{\left| \frac{\du \Rcal \left(u,v\right)}{\dv \Rcal \left(u,v\right)}\right| \, < \, \sqrt{a}\,.} $$

For the sequel, we define for any $N \,>\,0$ and $\alpha>0$,

\begin{equation}
\label{bound:00}
\left\{
\begin{array}{l}
 U(N,\alpha) := \left( 1 + \frac{1}{\sqrt{\alpha}}\right) \, N \,,
\\
\,
\\
 F(N,\alpha) := \underset{|\xi| \leq U(N,\alpha)}{\sup} |A(\xi)|\,,
\\
\,
\\
 V(N,\alpha) : = U(N,\alpha) \,\,+\,\,  \frac{1}{\sqrt{\alpha}} \,F\left(N,\alpha\right)\,. 
\end{array}
\right.
\end{equation}

We also denote  by $I(N,\alpha)$  the compact set  
\begin{equation}
\label{bound:01}
I(N,\alpha)\,:=\, 
\left[-\sqrt{a}\,V(N,\alpha),\, \sqrt{a}\, V(N,\alpha)\right]^2\,.
\end{equation}
Moreover, we assume that the initial conditions $\ueps_0,\veps_0$ are bounded independently of $\eps$ in $L^\infty(\R)$, such that:
\begin{equation}
\label{linf}
\displaystyle{N_0 : = \max \left\{ \underset{\eps > 0 }{\sup} \| \ueps_0\|_{\infty},  \underset{\eps > 0 }{\sup} \| \veps_0\|_{\infty} \right\} \, < \, \infty \,.}
\end{equation}
Consider any $a_0>0$  and assume that the function $\Rcal\in\C^1(\R\times\R, \R)$ satisfies (\ref{hyp:00}) and (\ref{hyp:01}).  We choose the characteristic speed $\sqrt{a}>0$ and the parameter $\beta>0$ such that
\begin{equation}
\label{hyp:03}
\left\{
\begin{array}{l}
\displaystyle{ \sqrt{a}\, > \, \max \left\{ \sqrt{a_0}\,,\, \frac{g\left(V(N_0, a_0)\right)}{\beta_0\left(V(N_0, a_0)\right)} \right\}\,,}
\\
\,
\\
\beta \,=\, h\left(V(N_0,a_0)\right), 
\end{array}\right.
\end{equation}
where $V$ is given  by (\ref{bound:00}).

\begin{remark}
\label{rem:00} 
Note that if we differentiate with respect to $u$ the equilibrium equation
$$\displaystyle{\Rcal (u,A(u)) = 0\,,} $$
we obtain
\begin{equation}
\label{sub:00}
\displaystyle{A'(u) = - \frac{\du \Rcal \left(u,A(u)\right)}{\dv \Rcal \left(u,A(u)\right)}}
\end{equation}
and thus recover the well known sub-characteristic condition in the case of semi-linear relaxation, namely:
$$\displaystyle{|A'(u)| \, < \, \sqrt{a}\,.} $$
\end{remark}

We present here the splitting Asymptotic Preserving scheme and its relaxed version. We introduce  a space time discretization based on a uniform grid of points $\displaystyle{(x_{j\,+\,1/2})_{j\, \in \, \Z} \, \subset \, \R}$, with space step $\Delta x$, and discrete time $t^n = n\, \Delta t$, $n\in\N$, for which the time step $\Delta t$ satisfies the CFL condition:
\begin{equation}
 \displaystyle{0 \, < \, \lambda := \frac{\sqrt{a}\, \Delta t}{\Delta x} \, < \, 1\,. }  
\label{CFL}
\end{equation}
We denote by $h = (\Delta t, \Delta x)$ the discretization parameter.

 \subsection{\bf An Asymptotic Preserving scheme for the relaxation system }
\label{AP:01}

In this section, we design a  numerical scheme for system (\ref{our:sys1})-(\ref{our:sys2}), by introducing a splitting between the linear transport part, and the nonlinear relaxation part, for which we will take advantage of the knowledge of the equilibrium (\ref{hyp:00}). In the asymptotic regime $\varepsilon\rightarrow 0$, the differential equation (\ref{our:sys1}) becomes stiff and explicit schemes are subject to a stability  constraints  $\Delta t = O(\eps)$. Of course, implicit schemes allow larger time steps, but new difficulty arises in computing the numerical solution of a fully nonlinear problem at each time step.  Here we want to combine both advantages of implicit and explicit schemes, that is, large time step for stiff problems and low computational complexity of the numerical solution at each time step. This is done, as said in the introduction, in the spirit of Asymptotic Preserving schemes for nonlinear relaxation problems introduced by F. Filbet \& S. Jin \cite{filb1} and G. Dimarco \& L. Pareschi \cite{DimarcoP}.

Thus we construct a numerical solution $(u_{h}^{\eps}, v_{h}^{\eps})$ to (\ref{our:sys1})-(\ref{our:sys2}) in  $\R^+\times \R$ as follows
\begin{equation}
\label{uh:vh}
\left\{ \begin{array}{ll}
 \displaystyle{u_{h}^{\eps} (t,x) }&\displaystyle{= \sum_{n\in\N} \,\sum_{j \in \Z}\, u_j^{n} \, \mathbf{1}_{C_j} (x)\, \mathbf{1}_{[t^{n}, t^{n+1}[} (t)  \,,}
\\
\;
\\
\displaystyle{v_{h}^{\eps} (t,x) }&\displaystyle{= \sum_{n\in\N} \,\sum_{j \in \Z}\, v_j^{n} \, \mathbf{1}_{C_j} (x)\, \mathbf{1}_{[t^{n}, t^{n+1}[} (t)  \,,}
\end{array}
\right.
\end{equation}
 where $C_j = ]x_{j-1/2}, x_{j+1/2}[$ are the space cells and the sequences $(u_j^{n})_{n,j}$ and $(v_j^{n})_{n,j}$ depend on $\eps$ and are given below.

First, initial data are computed as the averaged values of (\ref{our:sys2}) through each space cell: for all $j$ in  $\Z$,

\begin{equation}\label{sch:ini}
\left\{ \begin{array}{ll}
 \displaystyle{u_{j}^{0} }&\displaystyle{= \frac{1}{\Delta x} \, \int_{C_j} \, u_0^{\eps}(x) \, \ud x \,,}
\\
\,
\\
\displaystyle{v_{j}^{0} }&\displaystyle{= \frac{1}{\Delta x} \, \int_{C_j} \, v_0^{\eps}(x) \, \ud x \,.}
\end{array}
\right.
\end{equation}

Then, in order to discretize the system (\ref{our:sys1})-(\ref{our:sys2}), we apply a splitting strategy  into a linear transport part and a stiff ordinary differential part as follows. The first part consists to apply a time explicit scheme combined  with a finite volume method to the following linear differential system 
\begin{equation}\label{transport}
 \left\{ 
\begin{array}{l}
\displaystyle{\dt u + \dx v = 0 \,,}
\\
\,
\\
\displaystyle{\dt v + a\, \dx u = 0\,,}
\end{array}
\right.
\end{equation}
and then the second part deals with the stiff ordinary differential equations 
\begin{equation}\label{stiff}
 \left\{ 
\begin{array}{l}
\displaystyle{\dt u = 0 \,,}
\\
\,
\\
\displaystyle{\dt v  = -\frac{1}{\eps}\,\mathcal{R}(u,v)\,.} 
\end{array}
\right.
\end{equation}

We first approximate the linear transport  part  (\ref{transport}), that is,  for a given $\left(u^{n} \,, \, v^{n}\right)$, we compute  $(u^{n+1/2},\, v^{n+1/2})$ at time  $t^{n+1}$ with a standard Finite Volume scheme, that is, for all $j \, \in \, \Z$, 
\begin{equation}
\label{sch:00}
\left\{ 
\begin{array}{l}
\displaystyle{u_j^{n+1/2} \,\,=\,\, u_j^{n} \,\,-\,\, \Delta t \,D_h v_j^{n} \,,} 
\\
\,
\\
\displaystyle{v_j^{n+1/2} \,\,=\,\, v_j^{n} \,\,-\,\, \Delta t \,a\, D_h u_j^{n} \,,}
\end{array}
\right. 
\end{equation}
 where $D_h v_j^{n}$ and $a\,D_h u_j^{n}$ are discrete derivatives with respect to $x$ of $v$ and $u$, given for instance by a Lax-Friedrichs scheme, namely:
$$
\left\{
\begin{array}{l}
\displaystyle{D_h v_j^{n}  = \,\frac{1}{2\, \Delta x}\,\left[\left(v_{j+1}^{n} - v_{j-1}^{n}\right) - \sqrt{a}\,\left(u_{j+1}^{n} - 2\,u_j^{n} + u_{j-1}^{n}\right)\right]\,,}
\\
\,
\\
\displaystyle{a \,D_h u_j^{n} \, = \,\frac{1}{2\, \Delta x}\,\left[a\,\left(u_{j+1}^{n} - u_{j-1}^{n}\right) - \sqrt{a}\,\left(v_{j+1}^{n} -2\, v_j^{n} + v_{j-1}^{n}\right)\right]\,.}
\end{array}
\right.
 $$

\begin{remark}
\label{rem:01}
Of course, there is a wide range of possible choices for the numerical fluxes. As we will see below, the main property of the numerical scheme for the linear transport term that we require  is the TVD (\emph{Total Variation Diminishing}) property, namely, for all $n \, \in \N$,
$$\left\{
\begin{array}{l}
 \displaystyle{TV (u^{n+1/2}) \, \leq \, TV (u^{n})\,, }
\\
\,
\\
\displaystyle{TV (v^{n+1/2}) \, \leq \, TV (v^{n})\,, }
\end{array}
\right.
 $$
where  $\displaystyle{TV(u) \, : = \, \sum_{j\, \in \, \Z}\, \left|u_{j+1} - u_j\right|}$.
\end{remark}

Hence, the second part of the splitting only consists in approximating the nonlinear ordinary differential equation (\ref{stiff}), for all $j \, \in \, \Z$, starting from $(u_j^{n+1/2},v_j^{n+1/2})$. To this aim, we use the decomposition 
$$
\Rcal(u,v) = \left[\Rcal(u,v) - \beta\,\left(v -A(u)\right)\right] \,+\,  \beta\,\left(v - A(u)\right),
$$
where $\beta>0$ is a parameter such that
$$
0\,<\,\sup_{(u,v)}\partial_v \Rcal(u,v) \,\,<\,\, \beta.  
$$ 
Then, we apply  a time exponential scheme on the dissipative part and get the following numerical scheme:
\begin{equation}
\label{sch:01}
\left\{ 
\begin{array}{ll}
\displaystyle{u_j^{n+1}} =&\displaystyle{ u_j^{n+1/2} \,,} 
\\
\,
\\
\displaystyle{v_j^{n+1}} =& \displaystyle{ v_j^{n+1/2} - \left(v_j^{n+1/2} - A(u_j^{n+1/2})\right) \, \left[ 1 - \left(1+\frac{\beta \, \Delta t}{\eps}\right) e^{-{\beta \, \Delta t}/{\eps}} \right]} 
\\
\,
\\
 & \displaystyle{  - \frac{\Delta t}{\eps} \,  e^{-{\beta \, \Delta t}/{\eps}} \, \Rcal \left(u_j^{n+1/2}, v_j^{n+1/2}\right) \,.}
\end{array}
\right. 
\end{equation}

This numerical scheme (\ref{sch:00})-(\ref{sch:01}) allows to define a sequence $(u_j^n,v_j^n)_{(j,n)\in\Z\times\N}$.

 \subsection{\bf Convergence results }

We first establish a convergence result on the asymptotic behavior of the numerical solution to (\ref{sch:00})-(\ref{sch:01}) when $\eps$ tends to zero and $h=(\Delta t, \Delta x)$ is fixed.

\begin{theorem}
\label{theo:00}
Assume that the initial conditions $(\ueps_0,\veps_0)$ are bounded independently of $\eps$ in $BV(\R)$ and such that the assumption (\ref{linf}) is satisfied. Consider $\Rcal\in\C^1(\R\times\R, \R)$, which satisfies (\ref{hyp:00})-(\ref{hyp:01}) and the characteristic speed $\sqrt{a}>0$ and the parameter $\beta>0$ are given by (\ref{hyp:03}). Then, the solution $(u_{h}^{\eps}, v_{h}^{\eps})$ given by (\ref{uh:vh}) to the scheme (\ref{sch:00})-(\ref{sch:01}) with the initial data (\ref{sch:ini}), converges in  $L^1(\R)$, as $\eps \rightarrow 0$, to a numerical solution $(u_h,v_h)$, that is, 
$$
\| u^\eps_h(t) - u_h(t)\|_{1} + \| v^\eps_h(t) - v_h(t)\|_{1} \,\leq\,  C_t\,e^{-\beta_0\Delta t/\eps}\, \left[ 1+\left\|\delta^0\right\|_{1}\right],
$$
where  $(u_h,v_h)$  is a consistent approximation to the conservation laws (\ref{equi:sys1}) with   $v_h=A(u_h)$,
$$
u_h(t,x):=\sum_{n\in\N}\sum_{j\in\Z} \bar u_j^{n}\,\mathbf{1}_{C_j} (x)\, \mathbf{1}_{[t^{n}, t^{n+1}[} (t),
$$
and the sequence 
\begin{equation}
\label{copain:0}
\bar u_j^{n+1} = \bar u_j^n \,+\, \Delta t\, D_hA(\bar u_j^n), \qquad j\in\Z,\,\, n\geq 1, 
\end{equation}
with the initial data 
\begin{equation}\label{schLim:ini}
 \begin{array}{ll}
 \displaystyle{\bar u_{j}^{0} }&\displaystyle{= \frac{1}{\Delta x} \, \int_{C_j} \, u_0(x) \, \ud x \,,}
\end{array}
\end{equation}
where $u_0$ is given by (\ref{defu0}) and $\delta^0 = \ueps_0 - A(\veps_0)$.
\end{theorem} 

The proof of this theorem corresponds is analogous to the one corresponding to the continuous problem. The fact that we are able to establish uniform $BV$ bounds on the numerical solution allow to get error estimates with respect to $\eps$.

On the other hand, we propose a convergence result   of the Asymptotic Preserving scheme when $h=(\Delta t,\Delta x)$ goes to zero and $\eps$ is fixed:

\begin{theorem}
\label{theo:01}
Assume that the initial conditions $\ueps_0,\veps_0$ are bounded independently of $\eps$ in $BV(\R)$ and such that the assumption (\ref{linf}) is satisfied. Consider $\Rcal\in\C^1(\R\times\R, \R)$, which satisfies (\ref{hyp:00})-(\ref{hyp:01}) and the characteristic speed $\sqrt{a}>0$ and the parameter $\beta>0$ are given by (\ref{hyp:03}). Then, the solution $(u_{h}^{\eps}, v_{h}^{\eps})$ given by (\ref{uh:vh}) to the scheme (\ref{sch:00})-(\ref{sch:01}) and the initial data (\ref{sch:ini}), converges in  $L_{loc}^{1} \left(\R^+ \times \R\right)$, as $h \rightarrow (0,0)$, to a weak solution $(\ueps,\veps)$ to the relaxation Cauchy problem (\ref{our:sys1})-(\ref{our:sys2}). More precisely, we have 
$$
\|u_h^\eps(t) - u^\eps(t)\|_1 \,+\,\|v_h^\eps(t) - v^\eps(t)\|_1\,\leq\, \frac{C_t}{\eps} \left(\Delta t\,\left(\frac{\left\|\delta^{0}\right\|_{1}}{\eps} \,+\,  1\,\right) \,+\,\Delta x^{1/2} \right).   
$$
where $\delta^0 = \ueps_0 - A(\veps_0)$.
\end{theorem}

\section{\bf \emph{A priori} estimates}
\label{sec:apriori}
\setcounter{equation}{0}

We first define precisely the parameter $\beta$ and the assumption on the characteristic speed to ensure the stability of the scheme (\ref{sch:00})-(\ref{sch:01}) and  prove  estimates on the solution to the relaxation problem which are uniform with respect to $\eps$. In following  section, we drop  the subscripts $\eps$ for sake of clarity and investigate the stability property  of the Asymptotic Preserving scheme (\ref{sch:00})-(\ref{sch:01}).

\subsection{\bf \emph{A priori} estimate on the relaxation operator}

Let us focus on the second part of the scheme devoted to the approximation of the relaxation source term and give a technical lemma, which establishes a quasi-monotonicity property on the operator $G_{\eps,s}$ with $s>0$. 
 In order to do this, we will rather consider the equivalent formulation on the diagonal variables $w$ and $z$. Let us rewrite the splitting scheme on these variables. For $u$ and $v$ given, we set 
\begin{equation}
\label{def:wz}
\left\{
\begin{array}{l}
\displaystyle w \,=\, -v \,-\, \sqrt{a} \, u \,,
\\ 
\,
\\
\displaystyle z \,=\, +v \,-\, \sqrt{a}\, u\,.
\end{array}\right.
\end{equation}
Therefore, the linear transport scheme (\ref{sch:00}) written for $(w,z)$ exactly coincides with an  upwind finite volume method: for all $j\,  \in \, \Z$,
\begin{equation}
\label{sch:02}
 \left\{ 
\begin{array}{l}
\displaystyle w_j^{n+1/2} \,=\, w_j^n \,- \sqrt{a}\, \frac{\Delta t}{\Delta x}\, \left(w_{j}^n - w_{j-1}^n\right)  \,, 
\\
\,
\\
\displaystyle z_j^{n+1/2} \,=\, z_j^n \,+ \sqrt{a}\, \frac{\Delta t}{\Delta x}\, \left(z_{j+1}^n - z_{j}^n\right)  \,,
\end{array}
\right. 
\end{equation}
whereas the  nonlinear stiff part  (\ref{sch:01}) yields, for all $j \, \in \, \Z$
\begin{equation}
\label{sch:03}
\left\{ 
\begin{array}{l}
\displaystyle w_j^{n+1} \,=\, \displaystyle{ w_j^{n+1/2} \,+\, G_{\eps, \Delta t} \left(w_j^{n+1/2}, z_j^{n+1/2}\right) \,,}
\\
\;
\\
\displaystyle z_j^{n+1} \,=\, \displaystyle{ z_j^{n+1/2} \,-\, G_{\eps, \Delta t} \left(w_j^{n+1/2}, z_j^{n+1/2}\right) \,,} 
\end{array}
\right. 
\end{equation}
with  
\begin{eqnarray*}
\ G_{\eps, \Delta t} (w,z) &=& \displaystyle{ \left(\frac{z- w}{2} - A\left(-\frac{w + z}{2\, \sqrt{a}}\right)\right)\, \, \left[1 - \left(1+ \frac{\beta \, \Delta t}{\eps}\right)\,e^{-{\beta \, \Delta t}/{\eps}} \right]} 
\\
\,
\\
 &+&\displaystyle{ \frac{\Delta t}{\eps} \,  e^{-{\beta \, \Delta t}/{\eps}} \, \Rcal \left(-\frac{w + z}{2\, \sqrt{a}}, \frac{z - w}{2}\right) \,.}
\end{eqnarray*}

The main result of this section is to establish the quasi-monotonicity of the operator $G_{\eps,s}$, which will lead to $L^{\infty}$ an $BV$ estimates.

\begin{lemma}
\label{lemma:1}
Assume the function $\Rcal\in\C^1(\R\times\R, \R)$ satisfies (\ref{hyp:00})-(\ref{hyp:01}) and  choose $a>0$, $\beta>0$ such that  (\ref{hyp:03}) is verified. Then,
\begin{itemize}
\item[$(i)$]  the sub-characteristic condition  is satisfied for all $(w,z)\in I(N_0,a_0)$, that is,
\begin{equation}
\label{sub:01}
   \left| \frac{\du \Rcal}{\dv \Rcal}(u,v) \right| \, < \, \sqrt{a},
\end{equation}
where  $2\sqrt{a}\,u\,:=\,-(w+z)$ and $2\,v \,:=\,z-w$ ;
 
\item[$(ii)$] for all $\eps\,,\, s >0$, the source term operator $G_{\eps,s}$ is quasi-monotone  on the compact set $I(N_0,a_0)$, that is, 
\begin{equation}
\label{quasi:01}
\left\{
\begin{array}{l}
-1\,\leq \,\partial_w G_{\eps,\,s}(w,z) \, \leq \, 0,\,\,\, \forall\, (w, z)\,\in\,I(N_0,a_0), 
\\
\,
\\
\,\,\,0\,\leq\, \,\partial_z G_{\eps,\,s}(w,z) \, \leq \, 1,\,\,\, \forall\, (w, z)\,\in\,I(N_0,a_0);
\end{array}\right.
\end{equation}
\item[$(iii)$] consider for $i=1,2$,  $(w_i^{n+1}, z_i^{n+1})$ two solutions to (\ref{sch:03}) corresponding to 
two initial data $(w_i^{n+1/2}, z_i^{n+1/2})\in I(N_0,a_0)$. Then there exist $w$ and $z \in \R$ such that $|w|$, $|z|\leq \sqrt{a}\,V(N_0,a_0)$ and
\begin{equation}\label{quasi:02}
\left\{
\begin{array}{lll}
  w_1^{n+1} -w_2^{n+1} &=& \left(w_1^{n+1/2} - w_2^{n+1/2}\right)\, \left(1 + \dw G_{\eps,s} (w, z_1^{n+1/2})\right) 
\\
\,
\\
&+&  \left(z_1^{n+1/2} - z_2^{n+1/2}\right)\, \dz G_{\eps,s} (w^{n+1/2}_2, z)\,,
\\
\,
\\
 z_1^{n+1} -z_2^{n+1} &=& \left(z_1^{n+1/2} - z_2^{n+1/2} \right)\, \left(1 - \dz G_{\eps,s} (w^{n+1/2}_2, z)\right) 
\\
\,
\\
&-&  \left(w_1^{n+1/2} - w_2^{n+1/2}\right)\, \dw G_{\eps,s} (w, z_1^{n+1/2})\,.
\end{array}
\right.
 \end{equation}
\end{itemize} 
\end{lemma}

\begin{proof}
For any $N_0>0$ and $a_0>0$, we first observe that for $(w,z)\in I(N_0,a_0)$,  
$$
|u| \,=\, \frac{|w+z|}{2\sqrt{a}} \leq V(N_0,a_0).
$$
Therefore, using the assumption (\ref{hyp:01}) and the definition (\ref{bound:00}), we   get that
$$
\left|\frac{\du \Rcal}{\dv \Rcal}(u,v)\right| \, \leq \, \frac{g\left(V(N_0,a_0)\right)}{\beta_0\left(V(N_0,a_0)\right)}\, < \, \sqrt{a}\,,
$$
which proves the first assertion $(i)$.

Now we prove $(ii)$ the quasi-monotonicity property of $G_{\eps,s}$. Computing the partial derivatives of $G_{\eps,s}$, it yields for all $s>0$,
\begin{equation*}
\left\{
\begin{array}{l}
  \displaystyle{\dw G_{\eps,s}  \,=\, -\frac{1}{2}\left(1 - \frac{A'(u)}{\sqrt{a}}\right) \, \left[1- \left(1+\frac{\beta \, s}{\eps}\right)\,e^{-{\beta \, s}/{\eps}} \right] -  \frac{s}{2\,\eps}\,e^{-{\beta \, s}/{\eps}}\, \left(\frac{\du \Rcal}{\sqrt{a}} + {\dv \Rcal}\right)\,,}
\\
\,
\\
\displaystyle{\dz G_{\eps,s} \,=\, +\frac{1}{2} \left(1+ \frac{A'(u)}{\sqrt{a}}\right) \, \left[1- \left(1+\frac{\beta \, s}{\eps}\right)\,e^{-{\beta \, s}/{\eps}} \right] +  \frac{s}{2\,\eps}\,e^{-{\beta \, s}/{\eps}}\, \left(\frac{- \du \Rcal }{\sqrt{a}} + {\dv \Rcal}\right)\,.}
\end{array}
\right.
\end{equation*}

Hence,  from Remark~\ref{rem:00} and the sub-characteristic condition (\ref{sub:01}), we obtain that for all $(u,v)~\in~ I(N_0,a_0)$
$$ 
\dw G_{\eps,s}(w,z)  \,\,\leq\,\, 0\quad {\rm and }\quad \dz G_{a,s}(w,z) \,\,\geq\,\, 0\,.
$$
Moreover, still using condition (\ref{sub:01}), we also get  for all $(w,z)\in I(N_0,a_0)$
$$ \left\{
\begin{array}{l}
 \displaystyle{ \dw G_{\eps,s}(w,z) \,\,\geq\,\,  -\left[1 - \frac{\beta s}{\eps}\, e^{-\beta s/\eps} \right] \,-\, \partial_v \Rcal(u,v)\,\frac{s}{\eps} \, e^{-\beta s/\eps} \,,}
\\
\,
\\
\displaystyle{ \dz G_{\eps,s}(w,z) \leq  \left[1 - \frac{\beta s}{\eps}\, e^{-\beta s/\eps} \right] \,+\, \partial_v \Rcal(u,v)\,\frac{s}{\eps} \, e^{-\beta s/\eps}\,.}
\end{array}
\right.
$$

Now since  $|u|\,\leq\,V(N_0,a_0)$ and from the choice of the parameter $\beta$ in (\ref{hyp:03}), it yields that $|\partial_v \Rcal(u,v)|\leq \beta$. Therefore, we conclude that
$$
-1 \,\,\leq\,\, \dw G_{\eps,s}(w,z), \quad \dz G_{\eps,s}(w,z) \,\,\leq\,\, 1, \quad \forall(w,z)\in I(N_0,a_0).  
$$

Finally  $(iii)$ follows from  a first order Taylor expansion of $G_{\eps, s}$. 
\end{proof}
This Lemma allows to obtain the following comparison principle.

\begin{corollary}\label{coro:1}
Consider for $i=1,2$, two initial data $(w_i^{n+1/2}, z_i^{n+1/2})\in I(N_0,a_0)$ satisfying the monotonicity condition  
$$
w_1^{n+1/2} \leq w_2^{n+1/2} \quad {\rm and }\quad  z_1^{n+1/2} \leq z_2^{n+1/2}.
$$ 
Then,  the numerical solution $(w_i^{n+1},z_i^{n+1})$, given by (\ref{sch:03}) corresponding to the initial data $(w_i^{n+1/2},z_i^{n+1/2})$ for $i=1,2$,  satisfies
$$
w_1^{n+1} \leq w_2^{n+1} \quad {\rm and }\quad  z_1^{n+1} \leq z_2^{n+1}.
$$ 
\end{corollary}

\begin{proof}
 Starting from the equality (\ref{quasi:02}), it yields to the result applying the estimates (\ref{quasi:01}).
\end{proof}

\subsection{\bf $L^\infty$ estimates}

We establish a uniform bound on the numerical solution to the scheme (\ref{sch:00})-(\ref{sch:01} with respect to the time-space step $h = (\Delta t, \Delta x)$ such that (\ref{CFL}) is satisfied.

\begin{proposition}
\label{prop:01}
Consider any $a_0>0$ and 
$$
\displaystyle{N_0  = \max \left\{ \underset{\eps > 0 }{\sup} \| u_0\|_{\infty},  \underset{\eps > 0 }{\sup} \| v_0\|_{\infty} \right\} \,.}
$$
We assume that the function $\Rcal\in\C^1(\R\times\R, \R)$ satisfies (\ref{hyp:00})-(\ref{hyp:01}) and  choose $a>0$, $\beta>0$ such that  (\ref{hyp:03}) is verified. Then, for all $n\in\N$
$$
\|u^n\|_{\infty} \,\leq \, V(N_0,a_0), \quad \|v^n\|_{\infty} \,\leq \, \sqrt{a}\,V(N_0,a_0). 
$$
\end{proposition}
\begin{proof}
We proceed in two  steps and first construct  a particular solution $(\overline w^n,\overline z^n)\in \R^2$ to the scheme (\ref{sch:02})-(\ref{sch:03}) which is uniformly bounded, then we apply the comparison principle on the compact set $I(N_0,a_0)$ to prove an $L^\infty$ bound on $(u^n,v^n)$.

To this aim we choose $R_0=(1+\sqrt{a})\,N_0$ and  construct a numerical solution $(\overline w^n, \overline z^n)$ to (\ref{sch:02})-(\ref{sch:03}) corresponding to the initial data $(\overline w^0,\overline z^0)=(R_0,R_0)$, which does not depend on the space variable so that the transport step (\ref{sch:02}) is invariant. Then we apply the relaxation scheme (\ref{sch:03}), which yields 
$$
\overline w^n\,=\, -\overline v^n \,-\, \sqrt{a}\, \overline u^n,\quad 
\overline z^n\,=\, +\overline v^n\,-\, \sqrt{a} \,\overline u^n
$$  
where $(\overline u^n,\overline v^n)$ are only given by
$$
\left\{
\begin{array}{lll}
\overline u^n &=& \displaystyle{\overline u^0=-\frac{R_0}{\sqrt{a}}=\left(1+\frac{1}{\sqrt a}\right)\,N_0,}
\\
\,
\\
\overline v^n &=& \displaystyle{\left(1 + \frac{\beta\,\Delta t}{\eps}\right) \, e^{-\beta \Delta t/\eps}\, \overline v^{n-1} \,\,+\,\, \left(1-\left(1 + \frac{\beta\,\Delta t}{\eps}\right)\, e^{-\beta \Delta t/\eps} \right) A\left(\overline u^0\right)}
\\
\,
\\
&-& \displaystyle{\frac{\Delta t}{\eps}\, e^{-\beta \Delta t/\eps}\, \Rcal\left(\overline u^0,\overline v^{n-1}\right).}
\end{array}\right.
$$ 
Then, we proceed by induction to show that:
$$  \displaystyle \forall \, n \, \in \, \{0, \dots , N\}\,, \quad (\overline w^{n},\overline z^{n})\in I(N_0,a_0)\,. $$

We  assume that   $(\overline w^{n-1},\overline z^{n-1})\in I(N_0,a_0)$, for some $n\geq1$. Let us  prove that $(\overline w^{n},\overline z^{n})\in I(N_0,a_0)$. On the one hand since $\overline u^n=\overline u^0$,  it  yields 
$$
\|\overline u^n\|_{\infty} \,=\,\|\overline u^0\|_{\infty} \,\leq\, \left[  1\,+\,\frac{1}{\sqrt{a}} \right]\,N_0\,\leq \,U(N_0,a_0),
$$ 
where $U(N_0,a_0)$ is given by (\ref{bound:00}).

On the other hand, using a first order Taylor expansion of the source term $\Rcal(\overline u^0,.)$, we get that there exists $\tilde{v}^{n-1}\in \R$ such that 
\begin{eqnarray*}
\overline v^n &=& \left(1 + \frac{\beta-\partial_v\Rcal(\overline u^0,\tilde{v}^{k})}{\eps}\,\Delta t\right) \, e^{-\beta \Delta t/\eps}\,\, \overline v^{n-1} 
\\
&+& \left(1-\left(1 + \frac{\beta-\partial_v\Rcal(\overline u^0,\tilde{v}^{k})}{\eps}\,\Delta t\right)\, e^{-\beta \Delta t/\eps} \right) A\left(\overline u^0\right).
\end{eqnarray*}
Therefore,  denoting by $\lambda_k\in\R$, the real number such that
$$
\lambda_k:= \left(1 + \frac{\beta-\partial_v\Rcal(\overline u^0,\tilde{v}^{k})}{\eps}\,\Delta t\right)\, e^{-\beta \Delta t/\eps},\qquad \forall k\in\N,
$$ 
with $|\tilde{v}^{k}|\leq F(N_0,a_0)$, hence we get  
$$
\overline v^n = \lambda_{n-1} \,\overline v^{n-1} \,\,+\,\, (1-\lambda_{n-1})\, A(\overline u^0) 
$$
and since $\overline v^0=0$
$$
 \overline v^n \,\,=\,\, \left(1-\prod_{k=0}^{n-1}\lambda_k\right)\, A(\overline u^0).
$$
Moreover, using that $|\overline u^0|\leq U(N_0,a_0)$ and $\tilde{v}^{k}\leq \sqrt{a}\,V(N_0,a_0)$, for all $k\in\N$, we get from (\ref{hyp:01}) and (\ref{hyp:03}),
$$
0\,<\,\left(1 + \frac{\beta-\beta_0}{\eps}\,\Delta t\right)\, e^{-\beta \Delta t/\eps}\,\leq\, \lambda_k \,\leq\,\left(1 + \frac{\beta\,\Delta t}{\eps}\,\right)\, e^{-\beta \Delta t/\eps} \,<\,1,\quad \forall k\in\N.
$$
Therefore, $\|\overline v^n\|_{\infty} \leq F(N_0,a_0)$ and 
$$
\|\overline w^n\|_{\infty} \,, \,\, \|\overline z^n\|_{\infty} \,\,\leq\,\, F(N_0,a_0) \,+\, \,(1+\sqrt{a})\,N_0 \leq \sqrt{a}\,V(N_0,a_0),  
$$
that is, $(\overline w^n,\overline z^n)\in I(N_0,a_0)$.
 
Furthermore, starting from the following initial datum $(\underline w^0,\underline z^0)=(-R_0,-R_0)$, we can also construct another particular solution $(\underline w^n,\underline z^n)\in I(N_0,a_0)$ for all $n\in \{0, \dots, N\}$.

Now, we proceed to the second step which consists in applying the comparison principle of Corollary \ref{coro:1} to prove an $L^{\infty}$ estimate for any initial data $u^0$, $v^0\in L^\infty(\R)$ given by (\ref{sch:ini}). From the definition of $N_0$, we have
$$
 \|u^0\|_{\infty}, \,\,\|v^0\|_{\infty}\, \, \leq \, \, N_0.
$$ 
Then, we have for the initial data $(w^0,z^0)$ 
$$
\|w^0\|_{\infty}\,,\,\,\|z^0\|_{\infty} \,\leq \, (1+\sqrt{a})\,N_0 = R_0 \,\leq \sqrt{a}\,V(N_0,a_0).
$$ 
In other words, we have initially:
$$
\underline w^0\,\leq\, w^0 \,\,\leq\,\, \overline w^0,\quad \underline z^0\,\leq\,z^0\,\,\leq\,\, \overline z^0\,.
$$

Thus, we proceed by induction and assume that 
$$
\underline w^n\,\leq\, w^n \,\,\leq\,\, \overline w^n,\quad \underline z^n\,\leq\,z^n\,\,\leq\,\, \overline z^n\,.
$$
We first consider the linear transport step (\ref{sch:02}) to $(w^n,z^n)$ and get that
$$
\underline w^n\,\leq\, w^{n+1/2} \,\,\leq\,\, \overline w^n,\quad \underline z^n\,\leq\,z^{n+1/2}\,\,\leq\,\, \overline z^n\,.
$$
Thus, we apply Corollary \ref{coro:1} to the two solutions to (\ref{sch:03}) associated to the initial conditions
$(w_1^{n+1/2}, z_1^{n+1/2}) = (w^{n+1/2}, z^{n+1/2})$ and $(w_2^{n+1/2}, z_2^{n+1/2}) = (\underline w^n, \underline z^n)\, $ (and then  $(\overline w^n, \overline z^n)$), we have

$$
\underline w^{n+1}\,\leq\, w^{n+1} \,\,\leq\,\, \overline w^{n+1},\quad \underline z^{n+1}\,\leq\,z^{n+1}\,\,\leq\,\, \overline z^{n+1},
$$
which finally gives for all $n\in\N$, that $(w^n,z^n)\in I(N_0,a_0)$. By construction of $(u^n,v^n)$ we have proved that
$$
\|u^n\|_{\infty} \,\leq \, V(N_0,a_0), \quad \|v^n\|_{\infty} \,\leq \, \sqrt{a}\,V(N_0,a_0). 
$$
\end{proof}

\subsection{\bf $BV$ estimates}
In this section, we obtain a $BV$ estimate on the numerical solution to the scheme (\ref{sch:00})-(\ref{sch:01}) with the time-space step $h = (\Delta t, \Delta x)$ such that (\ref{CFL}) is satisfied.

\begin{proposition}
\label{prop:02}
Assume that $u_0,v_0$ are uniformly bounded with respect to $\eps$ in  $BV(\R)$. For any $a_0>0$ and $\displaystyle N_{0} = \max \left\{ \underset{\eps > 0 }{\sup} \| u_0\|_{\infty},  \underset{\eps > 0 }{\sup} \| v_0\|_{\infty} \right\}$, we assume that the function $\Rcal\in\C^1(\R\times\R, \R)$ satisfies (\ref{hyp:00})-(\ref{hyp:01}) and  choose $a>0$, $\beta>0$ such that  (\ref{hyp:03}) is verified. Then, for all $n \, \in \, \N$, we have:
$$
TV (w^{n+1}) \,+\,  TV (z^{n+1})  \,\, \leq \,\, TV (w^{n}) \,+\,TV (z^{n}). 
$$
\end{proposition}

\begin{proof}
First we note that $u^0$, $v^0\, \in \, BV(\R)$, then by construction,  $w^0$, $z^0\in \, BV(\R)$. To prove the $BV$ estimate, we proceed in two steps. On the one hand, using the $TVD$ property of the upwind scheme, we get that
$$
TV(w^{n+1/2}) \,\,\leq \,\, TV(w^n) \quad{\rm and}\quad TV(z^{n+1/2}) \,\,\leq \,\, TV(z^n).
$$ 
On the other hand,  we apply the nonlinear relaxation step (\ref{sch:03})    and  from Lemma~\ref{lemma:1} $(iii)$ with  $w_1^{n+1/2}= w^{n+1/2}_j$ , $z_1^{n+1/2}=z^{n+1/2}_j$ and  $w_2^{n+1/2}=w^{n+1/2}_{j+1}$, $z_2^{n+1/2} = z^{n+1/2}_{j+1}$, it yields for any $j\in\Z$,
$$
 |w_{j+1}^{n+1} -w_j^{n+1}| \,+\, |z_{j+1}^{n+1} -z_j^{n+1}| \,\;\leq\,\, |w_{j+1}^{n+1/2} -w_j^{n+1/2}| + |z_{j+1}^{n+1/2} -z_j^{n+1/2}|\,.
$$
Summing over $j \, \in \, \Z$, we get that
\begin{eqnarray*}
TV(w^{n+1}) \,+\, TV(z^{n+1}) &\leq& TV(w^{n+1/2})\,+\, TV(z^{n+1/2}) 
\\
&\leq & TV(w^{n})\,+\, TV(z^{n}).
\end{eqnarray*}
\end{proof}

\section{\bf Trend  to equilibrium (proof of Theorem~\ref{theo:00})}
\label{sec:deviation}
\setcounter{equation}{0}
For a sequence $u=(u_j)_{j\in\Z}$ we set
$$
\|u\|_{1} \,\,:=\,\, \sum_{j\in\Z} \Delta x\,|u_j|. 
$$ 
In this section we first focus on the asymptotic behavior of the numerical solution to (\ref{sch:00})-(\ref{sch:01}) when $\eps$ goes to zero or when times goes to infinity. Then, we prove that this numerical solution converges to a consistent  approximation of the conservation laws (\ref{equi:sys1}) when $\eps$ goes to zero. Roughly speaking, it corresponds to the limit $\PP^\eps_h\rightarrow \PP^0_h$, when $\eps\rightarrow 0$. 
 
\subsection{\bf Asymptotic behavior}

In this subsection, we drop the subscripts $\eps$ for sake of clarity and estimate the deviation to the local equilibrium $\delta^n = v^n - A(u^n)$.

\begin{proposition}
\label{prop:03}
Assume that $u_0,v_0$ are uniformly bounded with respect to $\eps$ in  $BV(\R)$. For any $a_0>0$ and $N_0$ as before, we assume that the function $\Rcal\in\C^1(\R\times\R, \R)$ satisfies (\ref{hyp:00})-(\ref{hyp:01}) and  choose $a>0$, $\beta>0$ such that  (\ref{hyp:03}) is verified. Then the deviation from the equilibrium, $\delta = v - A(u)$ satisfies for all $n \, \in \, \N$ and all $\eps>0$
\begin{equation}
\label{devi:00}
\left\{
\begin{array}{l}
\left\|\delta^{n+1/2}\right\|_{1} \,\leq\, \left\|\delta^n\right\|_{1} \,+\,  C\,\Delta t,
\\
\,
\\
\left\|\delta^{n}\right\|_{1} \,\leq\, e^{-\beta_0\,t^n/\eps} \,\left\|\delta^{0}\right\|_{1} \,+\,  \, C\,\eps.
\end{array}\right.
\end{equation}
where $C>0$ is a constant only depending  on the parameters $a$, $\beta_0$ and the $BV$ norm of the initial data.

Moreover, if $\eps< \Delta t$ then we get
\begin{equation}
\label{devi:01}
\left\|\delta^{n}\right\|_{1} \,\leq\, e^{-\beta_0\,t^n/\eps} \,\left\|\delta^{0}\right\|_{1} \,+\,  C_a\,\Delta t \,e^{-\beta_0\Delta t/\eps}.
\end{equation}
\end{proposition}

\begin{proof}
We set for $j \in \Z$, $n \in \N$ the sequence of  deviations from the equilibrium:
 $$\displaystyle{\delta_j^n = v_j^n - A\left(u_j^n\right)\,.}$$
  We first consider the transport step (\ref{sch:00}) of the numerical scheme: for all $j\, \in \, \Z$, we apply a Taylor expansion to $A$, then  there exists $\xi_j^n$ such that $\displaystyle{\left|\xi_j^n\right| \, \leq \, V(N_0,a_0)}$ and
\begin{eqnarray*}
\delta_j^{n+1/2} &=& \delta_j^{n} - \frac{\Delta t}{2\, \Delta x}\,\left[a\,\left(u_{j+1}^n - u_{j-1}^n\right) - \sqrt{a}\,\left(v_{j+1}^n -2\, v_j^n + v_{j-1}^n\right)\right]\\
&-& \frac{\Delta t}{2\, \Delta x}\,A'(\xi_j^n) \,\left[\left(v_{j+1}^n - v_{j-1}^n\right) - \sqrt{a}\,\left(u_{j+1}^n - 2\,u_j^n + u_{j-1}^n\right)\right]\,.
\end{eqnarray*}
Thanks to the uniform $BV$ estimate, proved in Proposition~\ref{prop:02},  the sub-characteristic condition
$$
\left|A'(\xi_j^n)\right| \, < \, \sqrt{a}\,, 
$$
and the TVD property of the numerical fluxes we get the first estimate (\ref{devi:00}), by multiplying by $\Delta x$ and summing over $j\in\Z$:
\begin{equation}
\label{res:1}
 \|\delta^{n+1/2}\|_{1} \, \leq \, \|\delta^n\|_{1} + \Delta t \, C_a\,\left[TV(v^0) + \sqrt{a}\,TV(u^0)\right]\,, 
\end{equation}
where $C_a>0$ is a constant only depending on $a$. 

Then, we consider the second step of the numerical scheme (\ref{sch:01}). On the one hand, since $u^{n+1}=u^{n+1/2}$, it yields
\begin{equation*}
\delta_j^{n+1} \,= \, \delta_j^{n+1/2} \, \left[ 1 \,+\, \beta \,\frac{\Delta t}{\eps}\right]\, e^{-{ \beta \, \Delta t}/{\eps}} \,-  \,\frac{\Delta t}{\eps} \,  e^{-{\beta \, \Delta t}/{\eps}} \, \Rcal (u_j^{n+1/2}, v_j^{n+1/2}) \,. 
\end{equation*}
On the other hand, applying a Taylor expansion, since $\Rcal(u,A(u))=0$ we get that there exists $\eta $ such that $\displaystyle{\left|\eta\right| \, \leq \, \sqrt{a}\,V(N_0,a_0)}$ and:
$$
\Rcal (u_j^{n+1/2}, v_j^{n+1/2})  \,=\, \dv \Rcal(u_j^{n+1/2},\eta) \,\delta_j^{n+1/2}.
$$
Hence, we have
$$
\delta_j^{n+1}  \,=\, \delta_j^{n+1/2} \, \left[1 \,+\, \left(1- \frac{\dv \Rcal(u_j^{n+1/2},\eta)}{\beta}\right) \,\frac{\beta\Delta t}{\eps}\right]\, e^{-{ \beta \, \Delta t}/{\eps}}.
$$

Therefore  under the assumption (\ref{hyp:01}), we set for all $s\geq 0$
$$
g(s) \,=\, \left[1 + \left(1-\frac{\beta_0}{\beta}\right) \,s \right]\, e^{-s}, 
$$
for which we easily show that for all $s\in \R^+$, we have that $e^{-s} \, \leq \,  g\left(s\right) \,\leq\, e^{-\beta_0\,s/\beta}$. Hence, for  $s=\Delta t/\eps$
\begin{equation}
\label{res:2} 
\left\|\delta^{n+1}\right\|_{1} \, \leq \, e^{-\beta_0\,\Delta t/\eps} \, \|\delta^{n+1/2}\|_{1}\,. 
\end{equation}
Finally, gathering (\ref{res:1}) and (\ref{res:2}), we obtain that there exists  a constant $C_1>0$  depending only on $a$, $TV(u^0)$ and $TV(v^0)$ such that
$$
\left\|\delta^{n+1}\right\|_{1}\, \leq \, e^{-\beta_0\,\Delta t/\eps} \, \left[\,\left\|\delta^{n}\right\|_{1} \,+\, C_1\, \Delta t \,\right]\,.  
$$
By induction, we easily get 
\begin{equation}
\label{peintre}
\left\|\delta^{n}\right\|_{1} \,\leq\,  e^{-\beta_0\,t^n/\eps} \,\left\|\delta^{0}\right\|_{1} \,\,+\, \, C_a\,\Delta t \,\frac{e^{-\beta_0\,\Delta t/\eps}}{1-e^{-\beta_0\,\Delta t/\eps}}. 
\end{equation}
To conclude we only observe that  ${x\,e^{-x}} \leq {1-e^{-x}}$,  for any $x\geq 0$, then it gives the second estimate of (\ref{devi:00}): there exists a constant $C>0$, only depending on $a$, $\beta_0$, $TV(u^0)$ and $TV(v^0)$ such that
$$ 
\left\|\delta^{n}\right\|_{1} \,\leq\, e^{-\beta_0\,t^n/\eps} \,\left\|\delta^{0}\right\|_{1} \,+\,  \, C\,\eps.
$$
Moreover, when $\eps< \Delta t$,  we again start from the estimate (\ref{peintre}) and note that   $1/(1-e^{-\beta_0\Delta t/\eps}) \leq 1/(1-e^{-\beta_0})$.
Thus, there exists another constant $C>0$, only depending on $a$, $\beta_0$, $TV(u^0)$ and $TV(v^0)$ such that
$$
\left\|\delta^{n}\right\|_{1} \,\leq\, e^{-\beta_0\,t^n/\eps} \,\left\|\delta^{0}\right\|_{1} \,+\,  C\,\Delta t \,e^{-\beta_0\Delta t/\eps},
$$
which gives (\ref{devi:01}).
\end{proof}

\subsection{\bf Proof of Theorem~\ref{theo:00}}

We are now ready to perform the asymptotic analysis of the numerical scheme (\ref{sch:00})-(\ref{sch:01}) when $\eps$ goes to zero. 

Let us consider the numerical solution $(u^\eps_h,v^\eps_h)$ to the scheme (\ref{sch:00})-(\ref{sch:01}) written in the form (\ref{sch:02})-(\ref{sch:03}) with 
$$
\left\{
\begin{array}{l}
w^\eps_h \,=\, -v^\eps_h \,-\, \sqrt{a} \, u^\eps_h \,,
\\ 
\,
\\
z^\eps_h \,=\, +v^\eps_h \,-\, \sqrt{a}\, u^\eps_h\,,
\end{array}\right.
$$
such that 
$$
\left\{
 \begin{array}{l}
\displaystyle w_h^\eps(t,x)=\sum_{n\in\N}\sum_{j\in\Z} w_j^{\eps,n}\,\mathbf{1}_{C_j} (x)\, \mathbf{1}_{[t^{n}, t^{n+1}[} (t), 
\\
\,
\\
\displaystyle z_h^\eps(t,x)=\sum_{n\in\N}\sum_{j\in\Z} z_j^{\eps,n}\,\mathbf{1}_{C_j} (x)\, \mathbf{1}_{[t^{n}, t^{n+1}[} (t).
\end{array}
\right.
$$
Let us also define $(w_h,z_h)$  the numerical solution to the scheme (\ref{sch:02})-(\ref{sch:03}) in the asymptotic limit $\eps=0$.  $w_h$ and $z_h$ are given by 
$$
\left\{
 \begin{array}{l}
\displaystyle w_h(t,x)=\sum_{n\in\N}\sum_{j\in\Z} w_j^{n}\,\mathbf{1}_{C_j} (x)\, \mathbf{1}_{[t^{n}, t^{n+1}[} (t), 
\\
\,
\\
\displaystyle z_h(t,x)=\sum_{n\in\N}\sum_{j\in\Z} z_j^{n}\,\mathbf{1}_{C_j} (x)\, \mathbf{1}_{[t^{n}, t^{n+1}[} (t),
\end{array}
\right.
$$
where $(w_j^n,z_j^n)$ is given by (\ref{copain:0}). To obtain an error estimates we rewrite the values $(w_j^n,z_j^n)_{(n,j)\in\N\times\Z}$ as a perturbation of the numerical solution to (\ref{sch:02})-(\ref{sch:03}) with a fixed value of $\eps>0$,  
$$
 \left\{ 
\begin{array}{l}
\displaystyle w_j^{n+1/2} \,=\, w_j^n \,- \sqrt{a}\, \frac{\Delta t}{\Delta x}\, \left(w_{j}^n - w_{j-1}^n\right), 
\\
\,
\\
\displaystyle z_j^{n+1/2} \,=\, z_j^n \,+ \sqrt{a}\, \frac{\Delta t}{\Delta x}\, \left(z_{j+1}^n - z_{j}^n\right),
\end{array}
\right. 
$$
and then 
\begin{equation}
\label{ref:031}
\left\{ 
\begin{array}{l}
w_j^{n+1} \;=\, \displaystyle{ w_j^{n+1/2} \,+\, G_{\eps, \Delta t}\left(w_j^{n+1/2}, z_j^{n+1/2}\right)  \,-\, \Delta t\,\Ecal_{j}^{n}(\eps),}
\\
\,
\\
z_j^{n+1} \,=\, \displaystyle{ z_j^{n+1/2} \,-\, G_{\eps, \Delta t}\left(w_j^{n+1/2}, z_j^{n+1/2}\right) \,+\, \Delta t\,\Ecal_{j}^{n}(\eps).}
\end{array}
\right. 
\end{equation}
where $\Delta t\,\Ecal_j^{n}(\eps)$ represents the consistency error of the operator $G_{\eps,\Delta t}$ with respect to $\eps$, that is,
$$
\Delta t\,\Ecal_{j}^{n}(\eps) := G_{\eps, \Delta t} \left(w_j^{n+1/2}, z_j^{n+1/2}\right)- G_{0, \Delta t} \left(w_j^{n+1/2}, z_j^{n+1/2}\right).
$$
Therefore, we apply Lemma~\ref{lemma:1} $(ii)$ and $(iii)$ in (\ref{ref:031}), with $(w_1,z_1)=(w_j^{\eps},z_j^{\eps})$ and $(w_2,z_2)=(w_j,z_j)$, it yields
\begin{eqnarray*}
|w_j^{\eps,n+1} - w_j^{n+1}| \,+\, |z_j^{\eps,n+1} - z_j^{n+1}|  &\leq& |w_j^{\eps,n+1/2} -w_j^{n+1/2}| \,+\, |z_j^{\eps,n+1/2} -z_j^{n+1/2}| 
\\
&+& 2\, \left|\Delta t\,\Ecal_{j}^{n}(\eps)\right|,
\end{eqnarray*}
and by linearity of the transport scheme (\ref{sch:02}), we have for all $n\geq 0$
$$
|w_j^{\eps,n+1} - w_j^{n+1}| \,+\, |z_j^{\eps,n+1} - z_j^{n+1}|  \,\leq\, |w_j^{\eps,n} -w_j^{n}| \,+\, |z_j^{\eps,n} -z_j^{n}| 
\,+\,2\,  \left|\Delta t\,\Ecal_{j}^{n}(\eps)\right|.
$$
Thus, multiplying by $\Delta x$, summing over $j\in\Z$ and applying a straightforward induction, we get the following stability result
\begin{eqnarray*}
\sum_{j\in\Z}\Delta x \left(|w_j^{\eps,n} - w_j^{n}| \,+\, |z_j^{\eps,n} - z_j^{n}|\right)  &\leq& \sum_{j\in\Z}\Delta x \left(|w_j^{\eps,0} -w_j^{0}| \,+\, |z_j^{\eps,0} -z_j^{0}|\right)
\\ 
&+& 2\,\sum_{k=0}^{n-1}\sum_{j\in\Z}\Delta t \,\Delta x\,\left|\Ecal_{j}^{k}(\eps)\right|.
\end{eqnarray*}
It now remains to evaluate the error $\Ecal^{n}_{j}(\eps)$. Using that for any $(u,v)\in I(N_0,a_0)$, the function $\Rcal\in\C^1(\R^+,\R)$ such that $\beta_0\leq\partial_v\Rcal(u,v)\leq \beta$, we have
\begin{eqnarray*}
\Delta t\,\left|\Ecal^{n}_{j}(\eps)\right| &=& e^{-{\beta \, \Delta t}/{\eps}}\,\left|-\left (v_j^{n+1/2} - A(u_j^{n+1/2})\right)\,\left(1+ \frac{\beta \, \Delta t}{\eps}\right) \,+\, \frac{\Delta t}{\eps}\Rcal \left(u_j^{n+1/2},v_j^{n+1/2}\right)\,\right|,
\\
&\leq & e^{-\beta_0\Delta t/\eps}\,\left|v_j^{n+1/2}- A(u_j^{n+1/2})\right|. 
\end{eqnarray*}
Thanks to the estimates (\ref{devi:00}) and (\ref{devi:01}) in Proposition~\ref{prop:03} on the deviation applied to  $v^{n+1/2}- A(u^{n+1/2})$ which is also valid in the asymptotic $\eps\rightarrow 0$, it yields 
$$
\|v^{n+1/2}- A(u^{n+1/2})\|_{1} \leq  
\left\{\begin{array}{ll}
\displaystyle \left\|\delta^0\right\|_{1}\,+\, C\,\Delta t & {\rm if }\,\, n =0,
\\
\,
\\
\displaystyle C\,\Delta t  & {\rm if }\,\, n > 0.
\end{array}\right.
$$
Then, we get for $k\geq 0$ and $\eps\leq \Delta t$,
$$
\sum_{j\in\Z}\Delta x \,\Delta t\,\left|\Ecal^{k}_{j}(\eps)\right| \,\,\leq\,\, \,e^{-\beta_0\Delta t/\eps}\, \left[ \left\|\delta^0\right\|_{1} \,+\, C\,\Delta t\right].
$$ 
Hence summing over $0\leq k \leq n$, it gives
$$
\sum_{k=0}^n\sum_{j\in\Z}\Delta x\,\Delta t\,\left|\Ecal^{k}_{\eps, \Delta t}\right| \,\,\leq\,\,  e^{-\beta_0\Delta t/\eps}\, \left[ \left\|\delta^0\right\|_{1} \,+\, C\,t^{n}\right].
$$ 
Finally, we get the estimate  
\begin{eqnarray*}
\| w^\eps_h(t^n) - w_h(t^n)\|_{1} + \| z^\eps_h(t^n) - z_h(t^n)\|_{1} &\leq&  \| w^\eps_h(0) - w_h(0)\|_{1} + \| z^\eps_h(0) - z_h(0)\|_{1} 
\\
&+&  2\,e^{-\beta_0\Delta t/\eps}\, \left[ {\left\|\delta^0\right\|_{1}} \,+\, C\,t^{n-1}\right]
\end{eqnarray*}
and the result follows  $(u_h^\eps,v_h^\eps)\rightarrow (u_h,v_h)$, when $\eps$ goes to zero. The proof of Theorem~\ref{theo:00} is now complete.

\section{\bf Proof of Theorem \ref{theo:01} }
\label{sec:convergence}
\setcounter{equation}{0}

In this section, we prove the convergence of the relaxation Asymptotic Preserving scheme.  As in the stability analysis of the relaxation scheme, we will rather consider the diagonal variables $w$ and $z$ and drop the subscripts $\eps$ for sake of clarity when it is not necessary.

\subsection{Consistency error}
Consider $(w,z)$ the exact solution to (\ref{our:sys1})-(\ref{our:sys2}) with (\ref{def:wz}). Unfortunately, this solution is not smooth enough to study the consistency error, then we introduce a regularization $(w_\delta,z_\delta)$  given by
$$
\left\{
\begin{array}{l}
w_\delta(t,x) = w\star \rho_\delta(t,x),
\\
\,
\\
z_\delta(t,x) = z\star \rho_\delta(t,x),
\end{array}\right.
$$
where $\star$ denotes the convolution product with respect to $x\in\R$ and 
$$
\rho_\delta(x)=\frac{1}{\delta}\rho\left(\frac{x}{\delta}\right)\quad {\rm and}\quad\rho\,\in\, \C^\infty_c(\R), \quad \rho \,\geq\, 0, \quad \int_{\R} \rho(z)\,\ud z \,=\, 1. 
$$
Thus, the couple $(w_\delta,z_\delta)$ is solution to 
\begin{equation}
\label{soldelta}
\left\{
\begin{array}{l}
\displaystyle\partial_t w_\delta \,+\, \sqrt{a}\,\partial_x w_\delta \,=\, +\frac{1}{\eps}\Rcal_\delta(u,v),
\\
\,
\\
\displaystyle\partial_t z_\delta \,-\, \sqrt{a}\,\partial_x z_\delta \,=\, -\frac{1}{\eps}\Rcal_\delta(u,v),  
\end{array}\right.
\end{equation}
with $\Rcal_\delta  = \Rcal\star \rho_\delta$ and $(u,v)$ solution to (\ref{our:sys1})-(\ref{our:sys2}). Therefore, applying the Duhamel's representation,  the solution can be written as
\begin{equation}\label{form_carac}
\left\{
\begin{array}{l}
\displaystyle w_\delta(t^{n+1},x) \,=\,w_\delta(t^n,x-\sqrt{a}\Delta t) \,+\,\frac{1}{\eps}\int_{0}^{\Delta t}\Rcal_\delta(u,v)(t^n+s,x-\sqrt{a}(\Delta t-s))\,\ud t,
\\
\,
\\
\displaystyle z_\delta(t^{n+1},x) \,=\,z_\delta(t^n,x+\sqrt{a}\Delta t) \,-\,\frac{1}{\eps}\int_{0}^{\Delta t}\Rcal_\delta(u,v)(t^n+s,x+\sqrt{a}(\Delta t-s))\,\ud t.
\end{array}\right.
\end{equation}
Then we set 
\begin{equation}
\label{wztilde}
\tilde w_j^n = \frac{1}{\Delta x} \,\int_{C_j} w_\delta(t^n,x)\,\ud x,\quad \tilde z_j^n = \frac{1}{\Delta x} \,\int_{C_j} z_\delta(t^n,x)\,\ud x.
\end{equation} 
Integrating (\ref{form_carac}) over $x\in  C_j$  and dividing by $\Delta x$, it yields
\begin{equation}
\label{consistance}
\left\{
\begin{array}{l}
\displaystyle\tilde w_j^{n+1}  =  \tilde w_j^{n+1/2}  \,+\, G_{\eps, \Delta t} \left(\tilde w_j^{n+1/2}, \tilde z_j^{n+1/2}\right) \, +\, \Delta t \,\Ecal_{1,j}^n \,+\, {\Delta t}\, \Ecal_{2,j}^n,
\\
\,
\\
\displaystyle\tilde z_j^{n+1}  =  \tilde z_j^{n+1/2}  \,-\, G_{\eps, \Delta t} \left(\tilde w_j^{n+1/2}, \tilde z_j^{n+1/2}\right) \, +\, \Delta t \,\Ecal_{3,j}^n \,+\, {\Delta t} \,\Ecal_{4,j}^n,
\end{array}\right.
\end{equation}
with
$$
\left\{
\begin{array}{l}
\displaystyle \tilde w_j^{n+1/2}  =  \tilde w_j^n \,- \sqrt{a}\, \frac{\Delta t}{\Delta x}\, \left(\tilde w_{j}^n - \tilde w_{j-1}^n\right), 
\\
\,
\\
\displaystyle \tilde z_j^{n+1/2}  =  \tilde z_j^n \,+ \sqrt{a}\, \frac{\Delta t}{\Delta x}\, \left(\tilde z_{j+1}^n - \tilde z_{j}^n\right). 
\end{array}\right.
$$ 
The consistency errors related to the transport operator $\Ecal_{1,j}^n$, $\Ecal_{3,j}^n$ are respectively defined by
$$
\displaystyle \Delta t\,\Ecal^n_{1,j} \,=\,\frac{\varepsilon^n_{1,j+1/2} - \varepsilon^n_{1,j-1/2}}{\Delta x},
\quad\displaystyle \Delta t\,\Ecal^n_{3,j} =\frac{\varepsilon^n_{3,j+1/2} - \varepsilon^n_{3,j-1/2}}{\Delta x},
$$
where $\varepsilon^n_{1,j+1/2}$ and $\varepsilon^n_{3,j+1/2}$ are the consistency errors of the numerical flux and are given by
\begin{equation}
\left\{
\begin{array}{l}
\displaystyle\varepsilon^n_{1,j+1/2} = -\int_{0}^{\sqrt{a}\Delta t} w_\delta(t^n, x_{j+1/2}-s)\ud s \,+\, \sqrt{a}\,\Delta t\,\tilde w^n_j,
\\
\,
\\
\displaystyle\varepsilon^n_{3,j+1/2} = +\int_{0}^{\sqrt{a}\Delta t} z_\delta(t^n, x_{j+1/2}+s)\ud s \,-\, \sqrt{a}\,\Delta t\,\tilde z^n_{j+1},
\end{array}\right.
\label{cop:03}
\end{equation}
whereas the consistency errors $\Delta t\,\Ecal^n_{2,j}$ and $\Delta t\,\Ecal^n_{4,j}$  correspond to the stiff source term and are given by
$$
\left\{
\begin{array}{l}
\displaystyle \Delta t\Ecal^n_{2,j} =  +\frac{1}{\Delta x}\int_{C_j} \int_{0}^{\Delta t}\frac{1}{\eps}\Rcal_\delta\left(u, v\right)(t^n+s,x-\sqrt{a}(\Delta t-s))\,\ud s - G_{\eps, \Delta t} \left(\tilde w_j^{n+1/2}, \tilde z_j^{n+1/2}\right)\ud x,
\\
\,
\\
\displaystyle \Delta t\,\Ecal^n_{4,j} =  -\frac{1}{\Delta x}\int_{C_j} \int_{0}^{\Delta t}\frac{1}{\eps}\Rcal_\delta\left(u, v\right)(t^n+s,x+\sqrt{a}(\Delta t-s))\ud s - G_{\eps, \Delta t} \left(\tilde w_j^{n+1/2}, \tilde z_j^{n+1/2}\right)\ud x.
\end{array}\right.
$$

We then  evaluate successively each consistency error term. On the one hand, we prove the following  consistency error for smooth solutions, which is related to the transport approximation.

\begin{proposition}
\label{prop:E1}
Let $(w,z)$ be given by (\ref{def:wz}), where $(u,v)$ is the exact solution to (\ref{our:sys1})-(\ref{our:sys2}) and such that  $w$, $z\in L^\infty(\R^+,BV(\R))$. Then the consistency error related to the transport part satisfies
$$
\sum_{j\in\Z} \Delta x \,\left[\, |\Ecal_{1,j}^n| \,+\,  |\Ecal_{3,j}^n|\,\right] \,\,\leq \,\, C\,\frac{\Delta x}{\delta}\,\left(\, TV(w(t^n))\,+\,TV(z(t^n)\,\right).
$$ 
\end{proposition}
\begin{proof}
We first study the consistency error for $w\in L^\infty(\R^+,BV(\R))$. We  perform a simple change of variable in (\ref{cop:03}), which yields  since  $\sqrt{a}\Delta t=\lambda\,\Delta x$,
\begin{eqnarray*}
\varepsilon^n_{1,j+1/2} &=&  -\lambda\int_{0}^{\Delta x} w_\delta(t^n,x_{j+1/2} - \lambda s)\ud s \,+\, \lambda\int_{0}^{\Delta x} w_\delta(t^n,x_{j+1/2}-s)\ud s,
\\
&=& \lambda \int_{0}^{\Delta x} \int_{\lambda s}^{s} {\partial_x w_\delta}(t^n,x_{j+1/2}-r)\ud r\,\ud s. 
\end{eqnarray*}
Therefore,  since $w_\delta$ is smooth we have
\begin{eqnarray*}
\left|\Ecal^n_{1,j}\right| &=& \frac{\sqrt{a}}{\Delta x^2}\,\left|\, \int_{0}^{\Delta x} \int_{\lambda s}^{s} \partial_x w_\delta(t^n,x_{j+1/2}-r) \,-\,\partial_x w_\delta(t^n,x_{j-1/2}-r)\ud r\,\ud s\,\right|, 
\\
&\leq&   \sqrt{a}\,\int_{x_{i-3/2}}^{x_{i+1/2}}   \left|\partial^2_{xx} w_\delta(t^n,x)\right|\,\ud x. 
\end{eqnarray*}
By multiplying by $\Delta x$ and summing over $j\in\Z$, we get an estimate  for a smooth solution $w_\delta(t^n)~\in~ W^{2,1}(\R)$,
$$
\sum_{j\in\Z}\Delta x \,\left|\Ecal^n_{1,j}\right| \,\,\leq \,\, 2\,\sqrt{a}\, \Delta x  \,\|\partial^2_{xx} w_\delta(t^n)\|_{1}.
$$
To achieve the proof, we need to estimate $\|\partial^2_{xx} w_\delta(t^n)\|_{1}$ with respect to $w$ and $\rho_\delta$. Using the convolution properties, we easily get 
$$
\|\partial^2_{xx} w_\delta(t^n)\|_{1} \,\leq \, \frac{C}{\delta}\,\|\partial_{x} w_\delta(t^n)\|_{1}\,\leq \, \frac{C}{\delta}\,TV(w(t^n)),
$$
which allows to conclude that
$$
\sum_{j\in\Z}\Delta x \,\left|\Ecal^n_{1,j}\right| \,\,\leq \,\, C\,\frac{\Delta x}{\delta}  \,TV(w(t^n)).
$$
Using a similar technique, we also get  for a smooth solution $z\in L^\infty(\R^+,BV(\R))$,
$$
\sum_{j\in\Z}\Delta x \,\left|\Ecal^n_{3,j}\right| \,\,\leq \,\, C\, \frac{\Delta x}{\delta} \,TV(z(t^n)).
$$
\end{proof}

On the other hand, we treat the consistency errors $\Ecal_{2,j}^n$ and $\Ecal_{4,j}^n$, which are related to the stiff source term.   

\begin{proposition}
\label{prop:E2}
Let $(w,z)$ be given by (\ref{def:wz}), where $(u,v)$ is the exact solution to (\ref{our:sys1})-(\ref{our:sys2}). Assume that  $w$, $z\in L^\infty(\R^+,BV(\R))$. Then there exists a constant $C>0$, only depending on $u$ and $v$ such that  the consistency error related to the stiff source part satisfies
$$
\sum_{j\in\Z}\Delta x\, \left|\Ecal^n_{2,j}\right| \,\,\leq\,\, C\,\left[\frac{\Delta t}{\eps}\,\left(\,e^{-\beta_0\,t^n/\eps} \,\frac{\left\|\delta^{0}\right\|_{1}}{\eps} \,+\,  1\,\right)\,+\,\frac{\Delta x}{\eps}  \,+\,\frac{\delta}{\eps} \right]   
$$
and
$$
\sum_{j\in\Z}\Delta x\, \left|\Ecal^n_{4,j}\right| \,\,\leq\,\, C\,\left[\frac{\Delta t}{\eps}\,\left(\,e^{-\beta_0\,t^n/\eps} \,\frac{\left\|\delta^{0}\right\|_{1}}{\eps} \,+\,  1\,\right)\,+\,\frac{\Delta x}{\eps}  \,+\,\frac{\delta}{\eps} \right].   
$$
\end{proposition}
\begin{proof}
We  first define $(\tilde u_j^n,\tilde v_j^n)$ such that 
$$
\tilde u_j^n = -\frac{\tilde w_j^n+\tilde z_j^n}{2\sqrt a}\quad {\rm and }\quad \tilde v_j^n = \frac{\tilde z_j^n-\tilde w_j^n}{2}.
$$ 
Therefore, we split the consistency error $\Ecal^n_{2,j}$  as
$$
\Ecal^n_{2,j} \,\,=\,\,  \Ecal^n_{21,j}\,+\,  \Ecal^n_{22,j}\,+\, \Ecal^n_{23,j} \,+\, \Ecal^n_{24,j}\,+\, \Ecal^n_{25,j},
$$
with 
$$
\left\{
\begin{array}{l}
\displaystyle \Delta t\Ecal^n_{21,j} = -\left[1 - \left(1+ \frac{\beta \, \Delta t}{\eps}\right)\,e^{-{\beta \, \Delta t}/{\eps}} \right] \,\left(\tilde v_j^{n+1/2} - A\left(\tilde u_j^{n+1/2}\right)\right),
\\
\;
\\
\displaystyle \Delta t\Ecal^n_{22,j} = \left(1-e^{-{\beta \, \Delta t}/{\eps}}\right) \,  \frac{\Delta t}{\eps} \,\Rcal(\tilde u_j^{n+1/2},\tilde v_j^{n+1/2}), 
\\
\,
\\
\displaystyle \Delta t\Ecal^n_{23,j} = \frac{1}{\eps\Delta x}\int_{C_j}\int_{0}^{\Delta t}\Rcal_\delta(u,v)(t^n+s,x-\sqrt{a}(\Delta t-s)) - \Rcal_\delta(u,v)(t^n,x-\sqrt{a}(\Delta t)) \ud s\ud x,
\\
\,
\\
\displaystyle \Delta t\Ecal^n_{24,j} = \frac{\Delta t}{\eps\,\Delta x}\int_{C_j}\Rcal_\delta(u,v)(t^n,x-\sqrt{a}(\Delta t)) - \Rcal(u,v)(t^n,x-\sqrt{a}(\Delta t))\ud x, 
\\
\,
\\
\displaystyle \Delta t\Ecal^n_{25,j} = \frac{\Delta t}{\eps\,\Delta x} \int_{C_j}\Rcal(u,v)(t^n,x-\sqrt{a}(\Delta t))- \Rcal\left(\tilde u^{n+1/2}_j,\tilde v_j^{n+1/2}\right)\,\ud x. 
\end{array}
\right.
$$
On the one hand, the two terms $ \Ecal^n_{21,j}$ and  $\Ecal^n_{22,j}$ can be easily evaluated using a Taylor expansion of $s\mapsto e^{-\beta s/\eps}$, it yields
$$
\Delta t\,|\Ecal^n_{21,j}| \,\leq\, \frac{1}{2}\left(\frac{\beta\Delta t}{\eps}\right)^2 \left|\tilde v_j^{n+1/2} - A\left(\tilde u_j^{n+1/2}\right) \right|. 
$$ 
Using that $\Rcal(u,A(u))=0$ and $\Rcal\in\C^1(\R^2,\R)$ with $\partial_v \Rcal(u,v)\leq \beta$, we also obtain that 
$$
\Delta t\,|\Ecal^n_{22,j}| \,\leq\, \left(\frac{\beta\Delta t}{\eps}\right)^2\, \left|\tilde v_j^{n+1/2} - A\left(\tilde u_j^{n+1/2}\right) \right|. 
$$ 
Therefore, from (\ref{devi:00}) in Proposition~\ref{prop:03}, we have
\begin{equation}
\label{res:e21}
\sum_{j\in\Z} \Delta x\,\left[\, |\Ecal^n_{21,j}| \,+\, |\Ecal^n_{22,j}| \,\right] \,\,\leq\,\, \,C\,\frac{\Delta t}{\eps} \, \left(\,e^{-\beta_0\,t^n/\eps} \,\frac{\left\|\delta^{0}\right\|_{1}}{\eps} \,+\,  1\,\right).
\end{equation}
On the other hand, we proceed to the evaluation of the terms $\Ecal^n_{23,j}$, $\Ecal^n_{24,j}$ and $\Ecal^n_{25,j}$. 
First, for $s\in [0,\Delta t]$, we set 
$$
\varphi_{\delta,x}(s) = \left[\Rcal(u,v)\star\rho_\delta\right](t^n+s,x-\sqrt{a}(\Delta t-s)).
$$
Then, from (\ref{hyp:01}) and (\ref{hyp:03}), we know that  $|\partial_u \Rcal(u,v)|\leq \sqrt{a}\,\beta$ and $|\partial_v \Rcal(u,v)|\leq \beta$, for any $(u,v)\in~ I(N_0,a_0)$, we obtain
\begin{eqnarray*}
\sum_{j\in\Z}\Delta x\Delta t\,\left|\Ecal^n_{23,j}\right| &\leq& \frac{1}{\eps}\int_{\R}\left|\int_{0}^{\Delta t}\int_0^{s} \varphi^\prime_{\delta,x} (\eta) \,\ud \eta \,\ud s\,\right|\ud x,
\\
&\leq& C\,\frac{\Delta t}{\eps}\,\int_{\R}\int_{t^n}^{t^{n+1}} \left(\left| \partial_t u_\delta\right| \,+\,\left|\partial_xu_\delta \right|\right)(t,x) \,\ud t \,\ud x
\\
&+& C\,\frac{\Delta t}{\eps}\,\int_{\R}\int_{t^n}^{t^{n+1}} \left(\left|\partial_t v_\delta\right|\,+\,\left|\partial_x v_\delta\right|\right)(t,x) \,\ud t \,\ud x.
\end{eqnarray*}
Thus we can use the estimates on the continuous relaxation system listed in Theorem \ref{theo_contin:1}. Indeed, since
\begin{equation*}
\left\{
\begin{array}{l}
\displaystyle\dt u_\delta \,=\,- \dx v_\delta,
\\
\,
\\
\displaystyle\dt v_\delta \, = \,-\, a\,\dx u_\delta  -\frac{1}{\eps}\Rcal_\delta(u,v),  
\end{array}\right.
\end{equation*}
we obtain, by applying a first order Taylor expansion of $\Rcal$, the inequalities
\begin{eqnarray*}
\int_{\R}\,\left(\left| \partial_t u_\delta\right| \,+\,\left|\partial_xu_\delta \right|\right)(t,x)\, \ud x &\leq& TV(u(t))\,+\, TV(v(t)), 
\\
\int_{\R}\, \left(\left|\partial_t v_\delta\right|\,+\,\left|\partial_x v_\delta\right|\right)(t,x) \,\ud x &\leq& C\, \left(TV(u(t)) + \frac{1}{\eps}\, \left\|(v-A(u))(t)\right\|_{1}\right).
\end{eqnarray*}
Hence, integrating over $t \in (t^n, t^{n+1})$ and using (\ref{resu:01}) and (\ref{resu:02}), we get:

\begin{equation}
\label{res:e23}
\sum_{j\in\Z}\Delta x\,\left|\Ecal^n_{23,j}\right| \,\,\leq\,\, C\,\frac{\Delta t}{\eps}\, \left(\,TV(u(t^n))\,+\, TV(v(t^n)) \,+\, \frac{e^{-\beta_0 t^n/\eps}}{\eps}\,\|{\delta^0}\|_{1} \,+\, 1\,\right),
\end{equation}
where $C>0$ only depends on $\sqrt{a}$ and $\beta$.

Now we treat the term $\Ecal^n_{24,j}$ using the smoothness properties  (\ref{hyp:01}) and (\ref{hyp:03}) on $\Rcal$, it gives 
\begin{eqnarray*}
\sum_{j\in\Z}\Delta x\,|\Ecal^n_{24,j}| &=& \frac{1}{\eps}\int_{\R}\left|\int_{\R}\left[\Rcal(u,v)(t^n,x-y-\sqrt{a} \Delta t) - \Rcal(u,v)(t^n,x-\sqrt{a} \Delta t)\right]\,\rho_\delta(y)\,\ud y\right|\,\ud x, 
\\
&\leq & \frac{C}{\eps} \int_{\R^2}\left[\left|u(t^n,x)-u(t^n,x-y)\right| \,+\,\left|v(t^n,x)-v(t^n,x-y)\right|\,\right]\rho_\delta(y)\,\ud y \,\ud x. 
\end{eqnarray*}
Thus, applying Fubini's theorem the $BV$ estimate on the exact solution (\ref{resu:01}) and the value of the integral of $\rho _{\delta}$,  we get 
\begin{equation}
\label{res:e24}
\sum_{j\in\Z}\Delta x\,\left|\Ecal^n_{24,j}\right| \,\,\leq\,\, C\,\frac{\delta}{\eps}\left[\, TV(u(t^n)) \,+\, TV(v(t^n)) \,\right].
\end{equation}
Finally, to deal with the last term $\Ecal_{25,j}^n$, we split it in two parts
\begin{eqnarray*}
\sum_{j\in\Z}\Delta x\,|\Ecal^n_{25,j}| &\leq& \frac{1}{\eps} \int_{\R}\left|\Rcal\left(u,v\right)(t^n,x-\sqrt{a}\Delta t) - \Rcal\left(u_\delta,v_\delta\right)(t^n,x-\sqrt{a}\Delta t)\right|\,\ud x
\\
&+& \frac{1}{\eps} \sum_{j\in\Z}\int_{C_j}\left|\Rcal\left(u_\delta,v_\delta\right)(t^n,x-\sqrt{a}\Delta t) - \Rcal\left(\tilde u^{n+1/2}_j,\tilde v_j^{n+1/2}\right)\right|\,\ud x
\end{eqnarray*}
and treat the different terms as for $\Ecal_{24,j}^n$, we get for the first one
$$
\int_{\R}\left|\Rcal\left(u,v\right)(t^n,x) - \Rcal\left(u_\delta,v_\delta\right)(t^n,x)\right|\,\ud x\,\leq\, C\,\delta\,\left[\, TV(u(t^n)) \,+\, TV(v(t^n)) \,\right]. 
$$
and for the latter one using the $BV$ estimate on the exact solution (\ref{resu:01}),
\begin{eqnarray*}
\sum_{j\in\Z}\int_{C_j}\left|\Rcal\left(u_\delta,v_\delta\right)(t^n,x-\sqrt{a}\Delta t) - \Rcal\left(\tilde u^{n+1/2}_j,\tilde v_j^{n+1/2}\right)\right|\,\ud x
\\
\,\leq\,  C\,\Delta x\,\left[\,\|\partial_x u_\delta(t^n)\|_{1} \,+\, \|\partial_x v_\delta(t^n)\|_{1},\right].
\end{eqnarray*}
Thus, we have
\begin{equation}
\label{res:e25}
\sum_{j\in\Z}\Delta x\,\left|\Ecal^n_{25,j}\right| \,\leq\, C\left(\frac{\delta}{\eps}+\frac{\Delta x}{\eps}\right)\,\left[\, TV(u(t^n)) \,+\, TV(v(t^n)) \, \right].
\end{equation}
Gathering (\ref{res:e21}), (\ref{res:e23}) , (\ref{res:e24}) and (\ref{res:e25}), and finally using the uniform bound on the BV norms of $(u,v)$ given in (\ref{resu:01}),  it yields
$$
\sum_{j\in\Z}\Delta x\, \left|\Ecal^n_{2,j}\right| \,\,\leq\,\, C\,\left[\frac{\Delta t}{\eps}\,\left(\,e^{-\beta_0\,t^n/\eps} \,\frac{\left\|\delta^{0}\right\|_{1}}{\eps} \,+\,  1\,\right)\,+\,\frac{\Delta x}{\eps} \,+\,\frac{\delta}{\eps} \right].   
$$
Using the same arguments we also prove that 
$$
\sum_{j\in\Z}\Delta x\, \left|\Ecal^n_{4,j}\right| \,\,\leq\,\, C\,\left[\frac{\Delta t}{\eps}\,\left(\,e^{-\beta_0\,t^n/\eps} \,\frac{\left\|\delta^{0}\right\|_{1}}{\eps} \,+\,  1\,\right)\,+\,\frac{\Delta x}{\eps} \,+\,\frac{\delta}{\eps} \right].   
$$
\end{proof}
\subsection{Convergence proof.}
Now we perform a rigorous analysis of the numerical scheme (\ref{sch:00})-(\ref{sch:01}) when $h=(\Delta t,\Delta x)$ goes to zero.  We consider the numerical solution $(u^\eps_h,v^\eps_h)$ to the scheme (\ref{sch:00})-(\ref{sch:01}) and $(u^\eps,v^\eps)$ the exact solution to (\ref{our:sys1})-(\ref{our:sys2}) and define $(w^\eps,z^\eps)$ using  (\ref{def:wz}). Then we denote by
$$
\bar w_j^{n} = \frac{1}{\Delta x} \,\int_{C_j} w^\eps(t^n,x)\,\ud x,\quad \bar z_j^n = \frac{1}{\Delta x} \,\int_{C_j} z^\eps(t^n,x)\,\ud x
$$
and $(w_j^n,z_j^n)_{(j,n)\in\Z\times\N}$ the numerical solution given by (\ref{sch:02})-(\ref{sch:03}). 
Thus,
\begin{eqnarray*}
\sum_{j\in\Z} \Delta x \left[\,|w_j^n-\bar w_j^n| \,+\, |z_j^n-\bar z_j^n|\;\right] &\leq&  \sum_{j\in\Z} \Delta x \left[\,|w_j^n-\tilde w_j^n| \,+\, |z_j^n-\tilde z_j^n|\;\right] 
\\
&+& \sum_{j\in\Z} \Delta x \left[\,|\tilde w_j^n-\bar w_j^n| \,+\, |\tilde z_j^n-\bar z_j^n|\;\right],
\end{eqnarray*}
where $(\tilde w_j^n,\tilde z_j^n)_{(j,n)\in\Z\times \N}$ is given by (\ref{wztilde}). On the one hand, we estimate the second terms of the right hand side using the convolution properties and have
\begin{equation}
\label{vitaly:00}
\sum_{j\in\Z} \Delta x \left[\,|\tilde w_j^n-\bar w_j^n| \,+\, |\tilde z_j^n-\bar z_j^n|\;\right]\,\leq\, C\,\delta\, \left[\,TV(u)\,+\,TV(v)\,\right].
\end{equation}
On the other hand, we apply the consistency error analysis to estimate the first term of the right hand side.  Applying (\ref{quasi:01})-(\ref{quasi:02}) established in Lemma~\ref{lemma:1} with $(\tilde w_j,\tilde z_j)$ and $(w_{j},z_j)$, it yields  
\begin{eqnarray*}
\sum_{j\in\Z}\Delta x\,|\tilde w_j^{n+1} - w_j^{n+1}|  &\leq&  \sum_{j\in\Z}\Delta x\,|\tilde w_j^{n+1/2}- w_j^{n+1/2}| \,\left( 1+\partial_wG_{\eps,\Delta t}(w_j,z_j^{n+1/2})\right)
\\
&+& \sum_{j\in\Z}\Delta x |\tilde z_j^{n+1/2} -z_j^{n+1/2}|\, \partial_zG_{\eps,\Delta t}(\tilde w_j^{n+1/2},z_j)
\\
&+&  \sum_{j\in\Z} \Delta x\,\Delta t\left[\, |\Ecal_{1,j}^n| \,+\, |\Ecal_{2,j}^n|   \,\right]
\end{eqnarray*}
and
\begin{eqnarray*}
\sum_{j\in\Z}\Delta x\,|\tilde z_j^{n+1} - z_j^{n+1}|  &\leq&  \sum_{j\in\Z}\Delta x\,|\tilde z_j^{n+1/2}- z_j^{n+1/2}| \,\left( 1-\partial_zG_{\eps,\Delta t}(\tilde w_j^{n+1/2},z_j)\right)
\\
&-& \sum_{j\in\Z}\Delta x |\tilde w_j^{n+1/2} -w_j^{n+1/2}|\, \partial_wG_{\eps,\Delta t}(w_j,z_j^{n+1/2})
\\
&+&  \sum_{j\in\Z} \Delta x\,\Delta t\left[\, |\Ecal_{3,j}^n| \,+\, |\Ecal_{4,j}^n|   \,\right].
\end{eqnarray*}
Summing the two inequalities and using that the scheme (\ref{sch:03}) is TVD, we get the following inequality
\begin{eqnarray*}
\sum_{j\in\Z}\Delta x\,\left[\, |\tilde z_j^{n+1} - z_j^{n+1}| \,+\, |\tilde w_j^{n+1} - w_j^{n+1}|\,\right]  &\leq&  \sum_{j\in\Z}\Delta x\,\left[\, |\tilde z_j^{n} - z_j^{n}| \,+\, |\tilde w_j^{n} - w_j^{n}|\,\right] 
\\
&+&  \sum_{j\in\Z} \Delta x\,\Delta t\left[\, |\Ecal_{1,j}^n| \,+\, |\Ecal_{2,j}^n|\,+\;|\Ecal_{3,j}^n| \,+\, |\Ecal_{4,j}^n|   \,\right].
\end{eqnarray*}
Therefore,
\begin{eqnarray*}
\sum_{j\in\Z}\Delta x\,\left[\, |\tilde z_j^{n+1} - z_j^{n+1}| \,+\, |\tilde w_j^{n+1} - w_j^{n+1}|\,\right]  \,\leq\,  \sum_{j\in\Z}\Delta x\,\left[\, |\tilde z_j^{0} - z_j^{0}| \,+\, |\tilde w_j^{0} - w_j^{0}|\,\right] 
\\
+\,  \sum_{k=0}^{n}\sum_{j\in\Z} \Delta x\,\Delta t\left[\, |\Ecal_{1,j}^k| \,+\, |\Ecal_{2,j}^k|\,+\;|\Ecal_{3,j}^k| \,+\, |\Ecal_{4,j}^k|   \,\right].
\end{eqnarray*}
Finally the consistency error analysis performed in Propositions~\ref{prop:E1} and \ref{prop:E2}, we have taking $\delta=\Delta x^{1/2}$
\begin{eqnarray}
\nonumber
\sum_{j\in\Z}\Delta x\,\left[\, |\tilde z_j^{n+1} - z_j^{n+1}| \,+\, |\tilde w_j^{n+1} - w_j^{n+1}|\,\right]  \leq  \sum_{j\in\Z}\Delta x\,\left[\, |\tilde z_j^{0} - z_j^{0}| \,+\, |\tilde w_j^{0} - w_j^{0}|\,\right] 
\\
+ \frac{C}{\eps}\,\left(\Delta t\,(\Delta t+\eps)\,\left(\frac{\left\|\delta^{0}\right\|_{1}}{\eps} \,+\,  1\,\right)   \,+\, t^n\,\left[\Delta x\,+\, \eps\Delta x^{1/2}  \,+\,\Delta x^{1/2} \right] \right).  
\label{vitaly:01}
\end{eqnarray}
Gathering (\ref{vitaly:00}) and (\ref{vitaly:01}), the right hand side converges to zero when $h=(\Delta t,\Delta x)$  goes to zero, which proves the convergence of the numerical solution (\ref{sch:00})-(\ref{sch:01}) to the exact solution to (\ref{our:sys1})-(\ref{our:sys2}).

Therefore, from the smoothness of the exact solution and the initial data $(u_0,v_0)$, it proves that for any discretization parameter $h$, for all $\eps > 0$:  
$$
\int_{\R}|u_h^\eps(t,x) - u^\eps(t,x)| \,+\,|v_h^\eps(t,x) - v^\eps(t,x)|\;\ud x \leq \frac{C}{\eps} \left(\Delta t\,\left(\frac{\left\|\delta^{0}\right\|_{1}}{\eps} \,+\,  1\,\right) \,+\,\Delta x^{1/2} \right).   
$$

\section{\bf Numerical simulations}
\label{sec:simul}
\setcounter{equation}{0}

\subsection{Generalized Jin\& Xin model}
We consider as a first numerical test the Jin \& Xin model with a nonlinear source term
\begin{equation}
\label{jinxin}
\left\{ 
\begin{array}{l}
\displaystyle{\dt u + \dx v = 0 \,,}
\\
\,
\\
\displaystyle{\dt v + a\, \dx u = -\frac{1}{\varepsilon} \,\mathcal{R}(u,v),}
\end{array}
\right.
\end{equation}
with $\Rcal(u,v)$ given by
$$
\Rcal(u,v) = \frac{v-u^2}{u^2+v^2}.
$$

We compute an approximation of the error for different meshes in space and time in order to evaluate the order of accuracy of the numerical scheme and different regimes corresponding to small and large values of the relaxation parameter $\eps>0$.  Therefore we choose two different initial data : the first one is a smooth initial datum given by  
\begin{equation}
\label{am:01} 
u_0(x)\,=\,\sin(2\pi\,x), \quad v_0(x)=0,\quad \forall x\in(0,1)
\end{equation}
whereas the second one is discontinuous
\begin{equation}
\label{am:02} 
u_0(x)\,=\,\, \left\{
\begin{array}{ll}
0.5  & \textrm{ if } x \leq 0,
\\
0.125  & \textrm{ if } x > 0,
\end{array}\right.
\end{equation}
and $v_0=0$.

We define an estimate of the numerical error  by  
$$
\mathcal{E}_1(h) \,=\, \|u_h^{\eps} - u_{2h}^{\eps}\|_1 \,+\, \| v_h^{\eps} - v_{2h}^{\eps}\|_1, 
$$ 
where $h=(\Delta x,\Delta t)$ is the discretization parameter and $\|\cdot\|_1$  is the discrete $L^1$ norm. In Figure~\ref{fig0a}, we present the curves corresponding to the order of accuracy with respect to $h$, computed from $\mathcal{E}_1(h)$ and observe that the order is relatively closed to one when the solution is smooth, whereas it decreases to $1/2$ when the solution is discontinuous which occurs when $\eps$ goes to zero or when the initial datum is discontinuous. Furthermore, as it has been proved in this paper, the numerical error is not too much affected by the variations the relaxation parameter $\eps$.

 \begin{center}
\begin{figure}[ht]
   \begin{tabular}{cc}
  \includegraphics[width =7.cm]{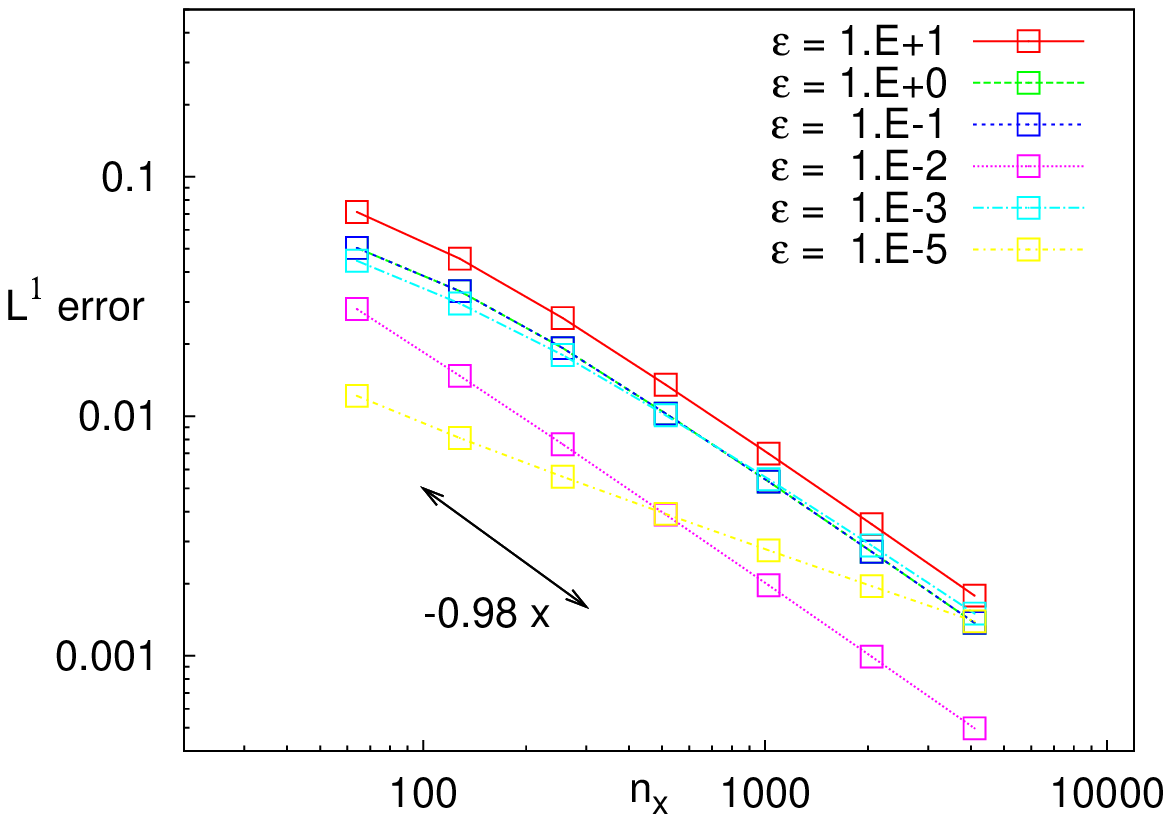}
  &
  \includegraphics[width =7.cm]{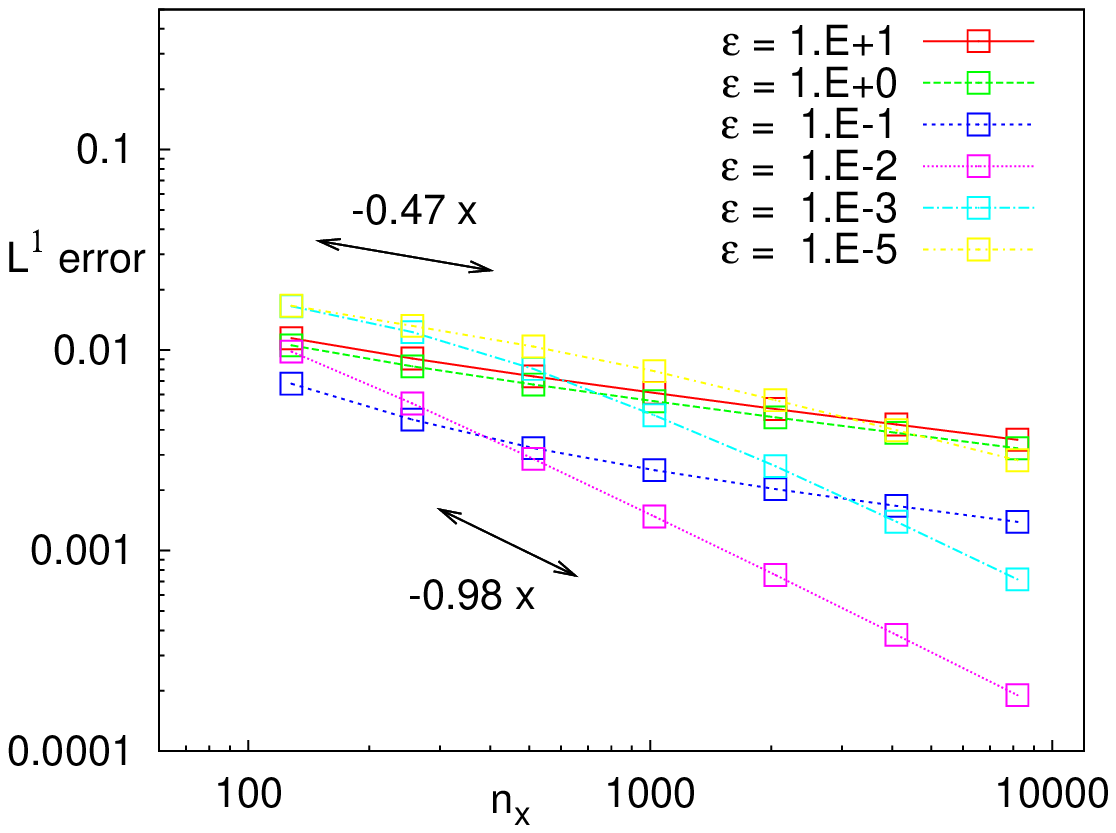}   
\\
$(a)$ & $(b)$
\end{tabular}
  \caption{Test 1 :  $L^1$ error {\it (a) smooth initial data (\ref{am:01}) and (b) discontinuous initial data (\ref{am:02})  for different regimes with initial data (\ref{ini2}).}}
  \label{fig0a}
 \end{figure}
 \end{center}

\subsection{The Broadwell model}
We now apply our scheme to the Broadwell model, which can be seen as a simple one-dimensional lattice Boltzmann equation, with only three discrete velocities \cite{broadwell, levB}.  The gas is defined by a density function in phase space satisfying the equation
\begin{equation}
\label{broadwell}
\left\{
\begin{array}{l}
\displaystyle \dt f_{+} + \dx f_{+} \,=\, \frac{1}{\eps}\, \left(f_0^2 \,-\,f_{+}\, f_{-}\right), 
\\
\,
\\
\displaystyle \dt f_0 \,= \, \frac{1}{2 \, \eps}\, \left(f_{+}\, f_{-} \,-\, f_0^2\right), 
\\
\,
\\
\displaystyle \dt f_{-} - \dx f_{-}\,=\, \frac{1}{\eps}\, \left(f_0^2\,-\,f_{+}\, f_{-}\right).
\end{array}
 \right.
\end{equation}
Here $f_+$, $f_-$ and $f_0$ denote the particle density distribution at time $t$, position $x$ with velocity $1$, $-1$ and $0$ respectively; whereas $\eps$ is the mean free path. 

We can rewrite the system with fluid dynamical variables. First, we define the density $\rho$, the momentum $m$ and $z$ as
\begin{equation}
\label{broad_kin}
\left\{
\begin{array}{l}
 \rho \,=\, f_+ \,+\, 2\, f_0 \,+\, f_-\,,
\\
m \,=\, f_+ \,-\, f_-\,,
\\
z \,=\, f_+ \,+\, f_-\,.
\end{array}
\right.
\end{equation}
 and the system (\ref{broadwell})  reads as
\begin{equation}
\label{broad1}
\left\{
\begin{array}{l}
\displaystyle \dt \rho + \dx m  \,=\, 0, 
\\
\,
\\
\displaystyle  \dt m + \dx z \,=\, 0\,, 
\\
\,
\\
\displaystyle \dt z + \dx m \,=\,  - \frac{1}{\eps}\,\left( \rho \, z - \frac{1}{2}\, \left(\rho^2 + m^2\right)\right)\,.
\end{array}
 \right.
\end{equation}
The local equilibrium corresponding to (\ref{broad1}) is then defined by $z = A\left(\rho,m\right)$, where 
$$
A\left(\rho,m\right)  = \frac{1}{2}\, \left(\rho + \frac{m^2}{\rho}\right). 
$$
Hence, when $\eps$ goes to zero, we obtain the following system of equations
\begin{equation}
\label{euler}
\left\{
\begin{array}{l}
\displaystyle\dt \rho + \dx m \,=\;0,
\\
\,
\\
\displaystyle\dt m + \dx A\left(\rho,m\right) \,=\, 0.
\end{array}\right.
\end{equation}
We still perform several numerical simulations on this model in order to verify the order of accuracy and the asymptotic behavior of the numerical solution for different values of $\eps>0$.

\paragraph{\it Test 1 : Approximation of smooth solutions.} For this numerical test, we consider a smooth initial data \cite{Par-Ru}: 
\begin{equation}\label{ini2}
 \left\{\begin{array}{l}
\rho_0(x) = \rho_g + 0.2\, \sin(\pi \, x),
\\
\,
\\
  m_0(x) = m_g,
\\
\,
\\
\displaystyle z_0(x) = \frac{1}{2}\, \left(\rho_0 + \frac{m_0^2}{\rho_0}\right),
 \end{array}\right.
\end{equation}
which is initially at the local equilibrium  $z_0=A\left(\rho_0,m_0\right)$. We apply periodic boundary condition in space and a CFL number $\lambda = 0.9$.

On the one hand, we compute as in the previous case an estimate of the numerical error for different meshes in space and time in order to evaluate the order of accuracy of the  scheme with respect to $h$ and for different regimes. In Figures \ref{fig3a}, we present the plots in log scale of the error estimate given by
$$
\mathcal{E}_p \,=\, \|\rho_h^{\eps} - \rho_{2h}^{\eps}\|_p \,+\, \| m_h^{\eps} - m_{2h}^{\eps}\|_p \,+\, \| z_h^{\eps} - z_{2h}^{\eps}\|_p, 
$$ 
where $h=(\Delta x,\Delta t)$ is the discretization parameter and $\|\cdot\|_p$  is the discrete $L^p$ norm with $p\in\{1,2,\infty\}$. We observe that in this case the order of accuracy of our method is relatively closed to one both for $L^1$ and $L^\infty$ norms and the numerical error is not affected by the large variations of the relaxation parameter $\eps$.

 \begin{figure}[ht]
  \begin{center}
  \begin{tabular}{cc}
  \includegraphics[width =7.cm]{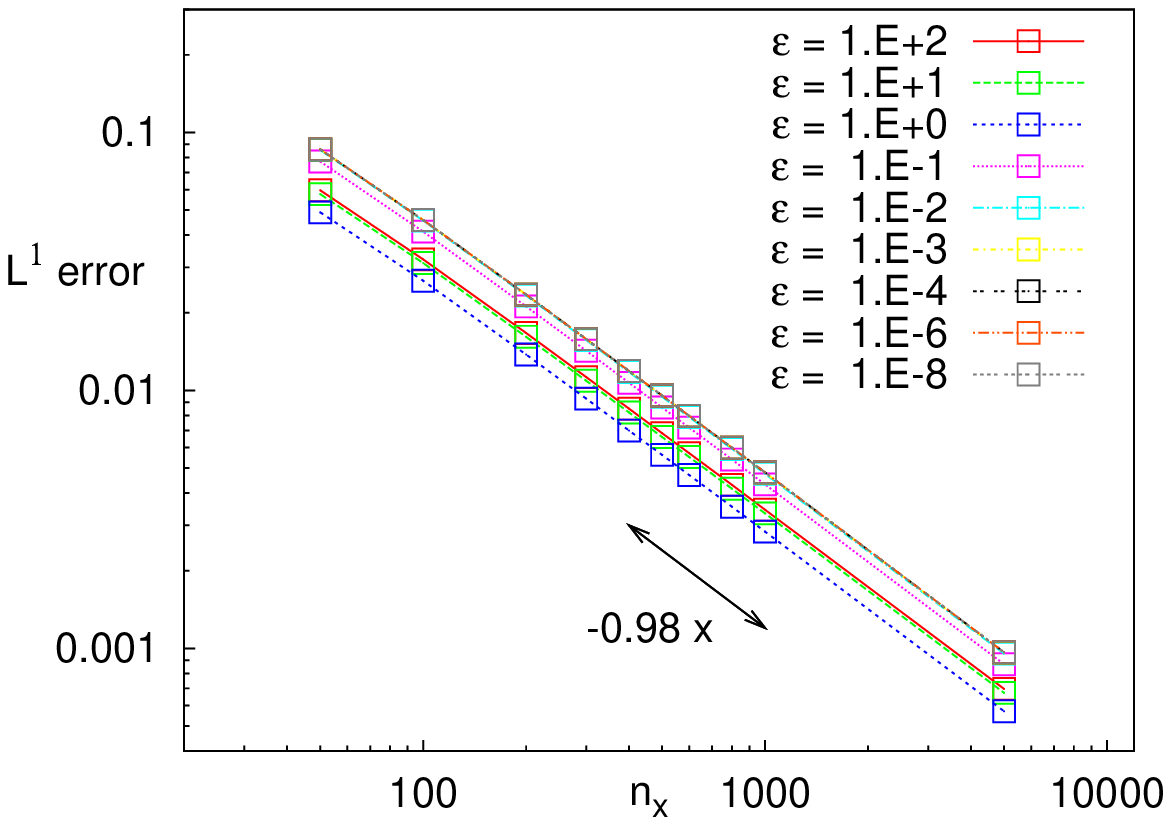}
  &
  \includegraphics[width =7.cm]{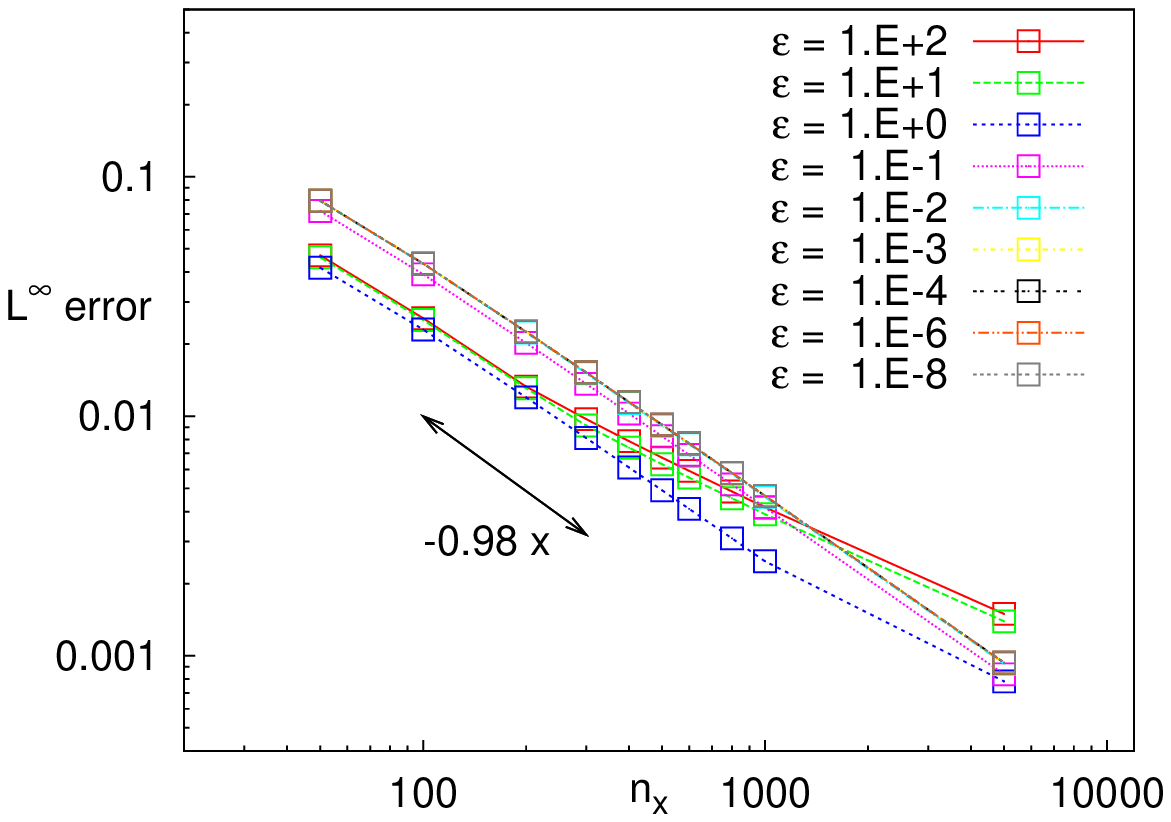}   
\\
$(a)$ & $(b)$
\end{tabular}
  \caption{Test 1 : Approximation of smooth solutions.  {\it (a) $L^1$ error and (b) $L^\infty$ error at time $t=1$ for different regimes with initial data (\ref{ini2}).}}
  \label{fig3a}
  \end{center}
 \end{figure}

On the other hand, we investigate numerically the long-time behavior of the solution to the Broadwell system (\ref{broadwell}). We consider $f = (f_+, f_0, f_-)^T$ a smooth solution to (\ref{broadwell}) with periodic boundary conditions in space. This solution is expected to converge to the global equilibrium $(\rho_g,m_g)$ as $t$ goes to $+ \infty$. In order to observe this trend to equilibrium, we investigate the behavior of the quantities
$$
\mathcal{S}_\rho(t) \,=\, \|\rho(t) - \rho_g\|_{1}, \quad \mathcal{S}_m(t) \,=\, \|m(t) - m_g\|_{1}, 
$$
where the initial data is a perturbation of this global equilibrium is $\rho_g = 1$, $m_g = 0$ and $z_g = A\left(\rho_g,m_g\right)$. In Figure \ref{fig4}, we present the time evolution of $\mathcal{S}_\rho(t)$, $\mathcal{S}_m(t)$ for the relaxation model and the hydrodynamic limit (\ref{euler}). Thus, we observe that the frequency of oscillations does not depend on $\eps$, whereas the  damping rate  is proportional to $\eps$. In other words, it appears that the trend to equilibrium is much faster (exponentially fast) in the rarefied regime, which corresponds to large values of $\eps>0$ than  in the hydrodynamic one. In the hydrodynamic regime $\eps=0$, the solution does not converge to the equilibrium since dissipative terms disappear. 
 \begin{center}
\begin{figure}[ht]
  \begin{tabular}{cc}
  \includegraphics[width = 7.cm]{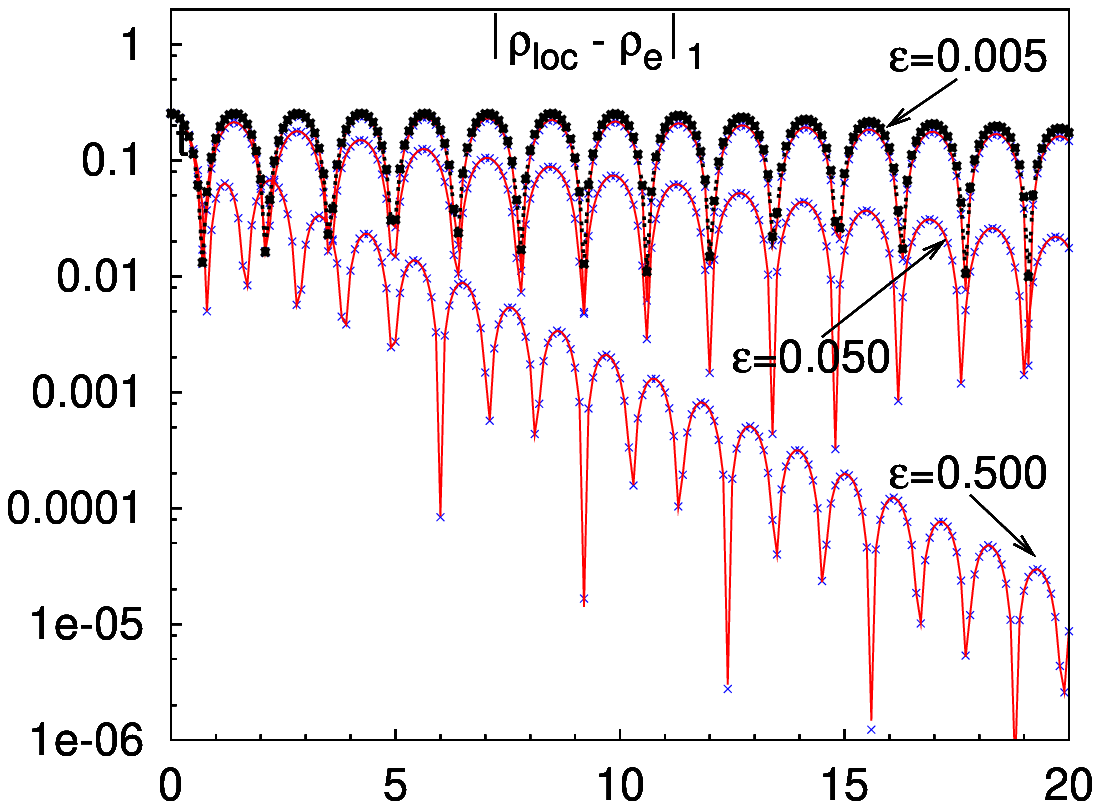}
  &
  \includegraphics[width = 7.cm]{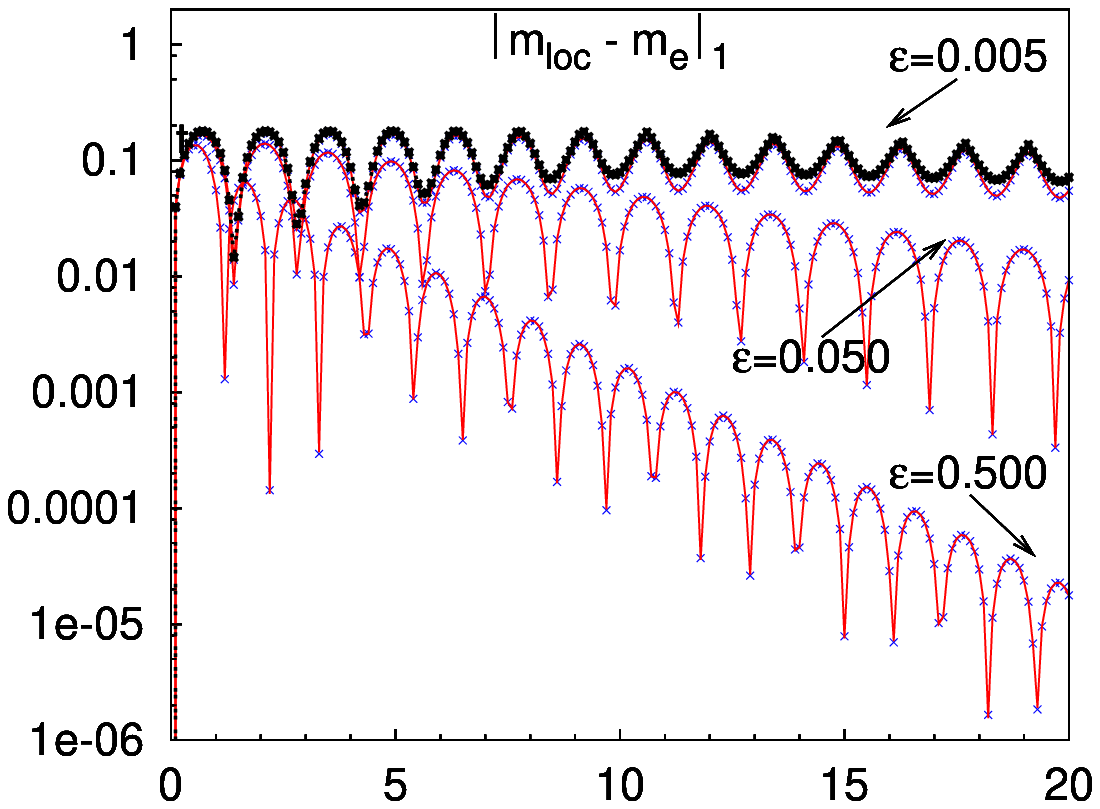} 
  \end{tabular}
  \caption{Test 1 : Approximation of smooth solutions. {\it Influence of the relaxation parameter $\eps$: evolution of $\mathcal{S}_\rho$ (left) and $\mathcal{S}_m$ (right) for $\eps =  0.5$, $0.05$, $0.005$ and comparison with the hydrodynamic limit (\ref{euler}) (dashed black line). }}
  \label{fig4}
 \end{figure}
 \end{center}

\paragraph{\it Test 2 : The Riemann problem.} We now consider a discontinuous initial datum and propose numerical simulations to (\ref{broadwell}) with different relaxation parameters $\eps$, from rarefied to fluid regimes. The initial condition is given by the local equilibrium 
\begin{equation}
\label{ini1}
  (\rho_0,m_0, z_0) = 
 \left\{
 \begin{array}{ll}
  (1, 0, 1), & {\rm if} \,-1 \leq x \leq 0,
 \\
 \,
 \\
 (0.25, 0, 0.125) ,& {\rm if }\, 0 < x \leq 1.
  \end{array}\right.
\end{equation}
 We approximate the Broadwell system in $(-1, 1)$, with reflecting boundary condition and a CFL number $\lambda = 0.9$ and present the evolution of the numerical solution. We only use $100$ grid points, so that the time step is fixed equal to $\Delta x=0.002$, and compare the results with the ones obtained from a fully explicit solver for which the time step has to be of the order of $\eps$ for different values of the relaxation parameter.

We first compute an approximation of the $L^p$ error with different meshes in space and time in order to evaluate the order of accuracy.  In Figure \ref{fig2}, we present the plots of $\mathcal{E}_1(h)$ and $\mathcal{E}_2(h)$  in log scale where $h=(\Delta x,\Delta t)$ is the discretization parameter. In that case, the order of accuracy is not fixed and depends on the different regimes, but in the worst case the numerical error  $\mathcal{E}_1(h)$ is of order $\sqrt{h}$, as in the  estimate (\ref{vitaly:01}). 

\begin{figure}[ht]
  \begin{center}
  \begin{tabular}{cc}
\includegraphics[width =7.cm]{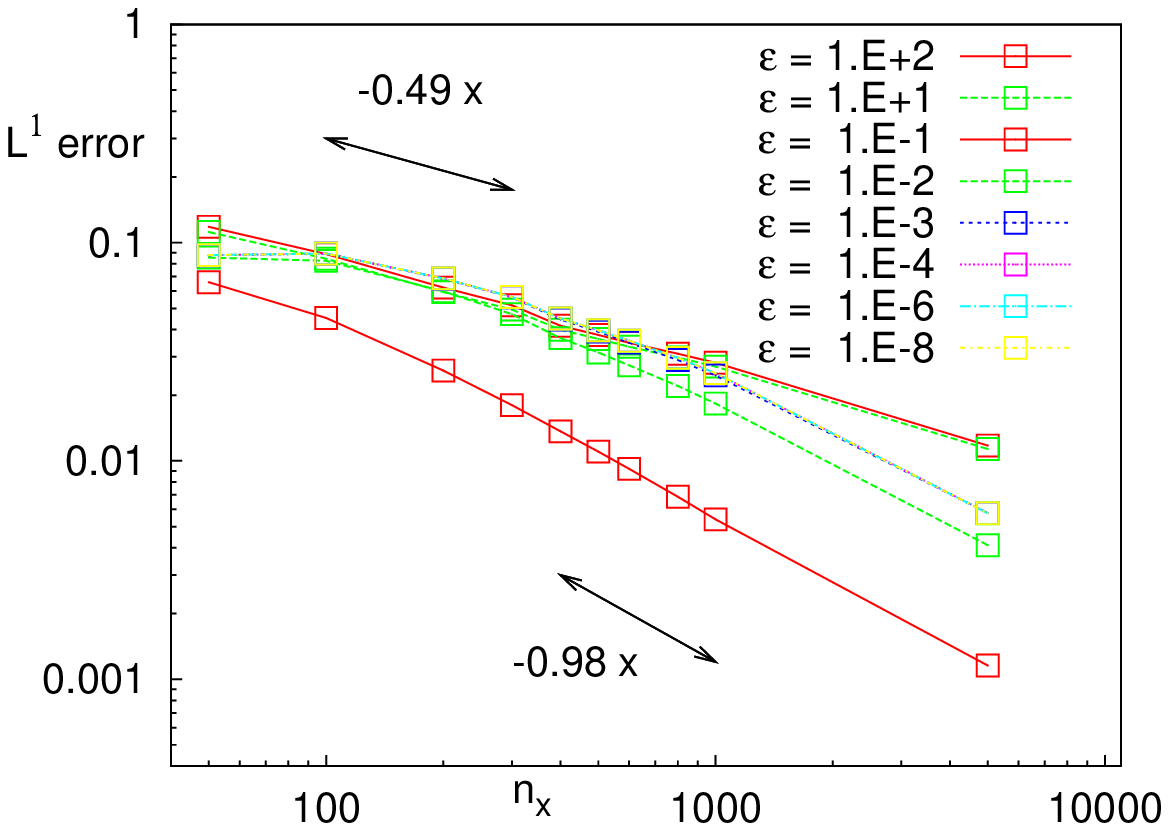}
  &
\includegraphics[width =7.cm]{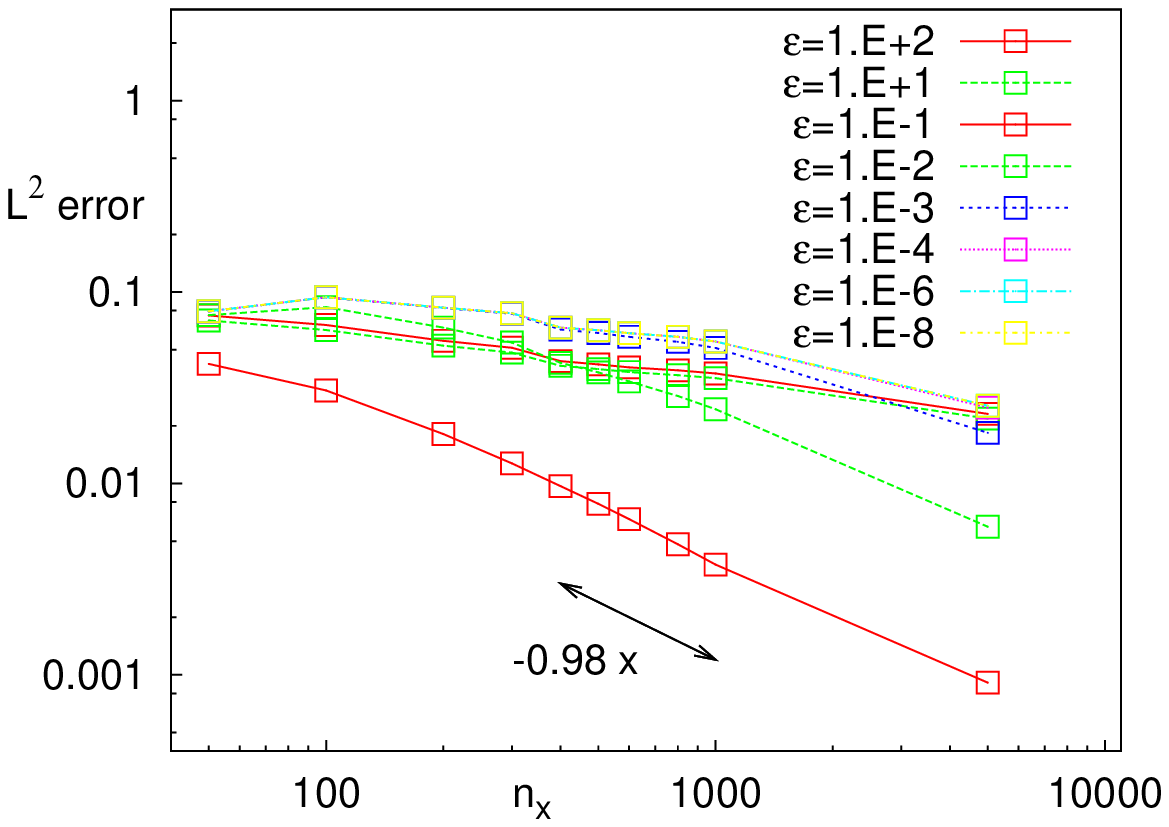}   
\\
$(a)$ & $(b)$
    \end{tabular}
  \caption{ Test 2 : The Riemann problem. {\it  (a) $L^1$ error and (b) $L^2$ error at time $t=1$ for different regimes, with initial data (\ref{ini1}). }}
  \label{fig2}
  \end{center}
 \end{figure}

In Figure \ref{fig1}, we take $\eps = 0.5$ and $0.1$. For such values of $\eps$, the problem is not stiff and this test is performed to compare the accuracy of our AP scheme with a fully explicit scheme (global Lax-Friedrichs method with slope limiters and explicit Euler discretization in time).
The density $\rho$, the momentum $(m,z)$, and the deviation to the equilibrium $z - A\left(\rho,m\right)$ are plotted at different times $t = 0.05$, $0.2$, $0.35$ and $0.5$. At the kinetic regime, we can observe that our method gives the same accuracy as a standard fully explicit scheme.

\begin{figure}[ht]
\begin{center}
\begin{tabular}{cc}
\includegraphics[width=7.cm]{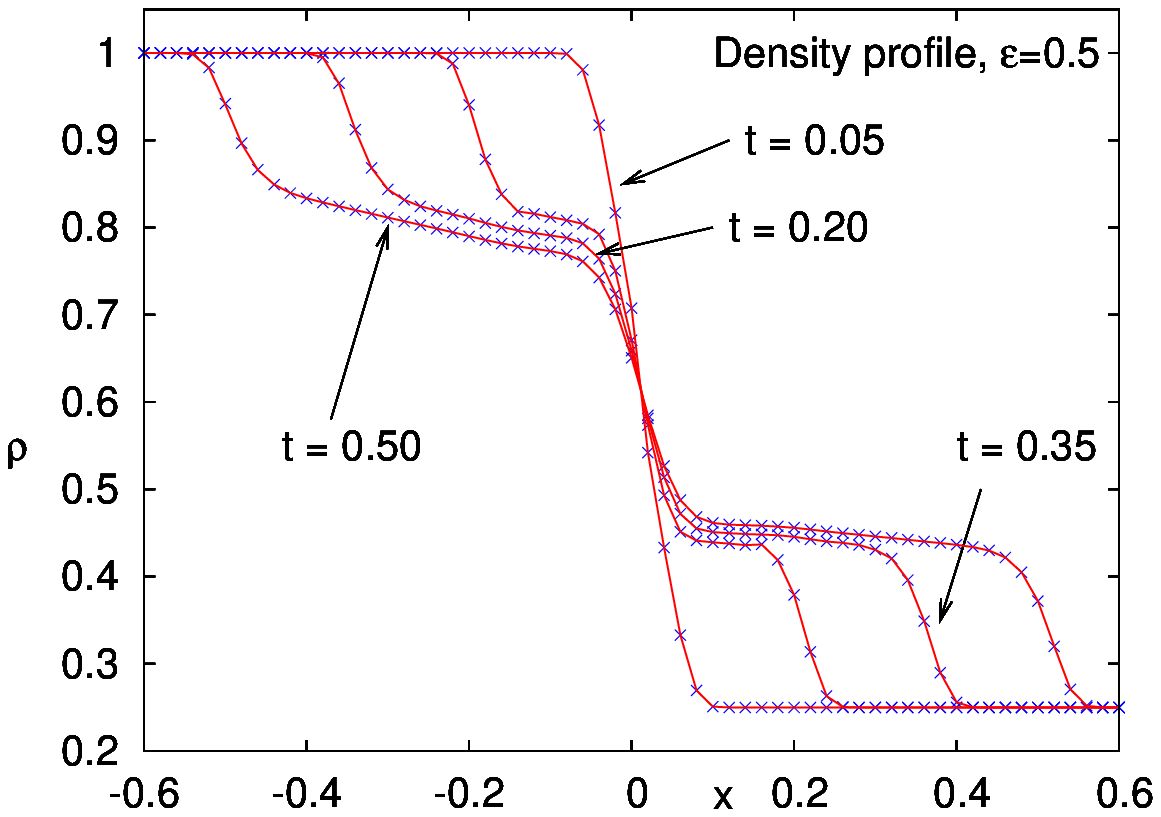}&
\includegraphics[width=7.cm]{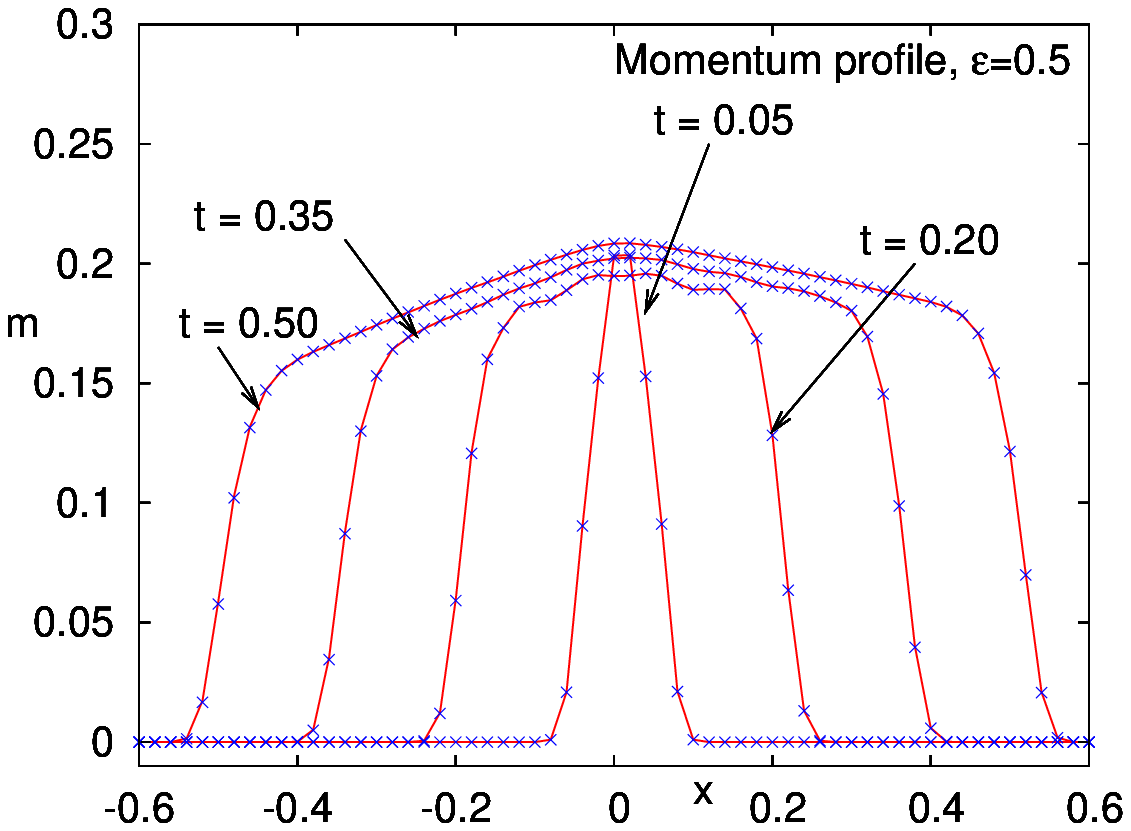} 
\\
\includegraphics[width=7.cm]{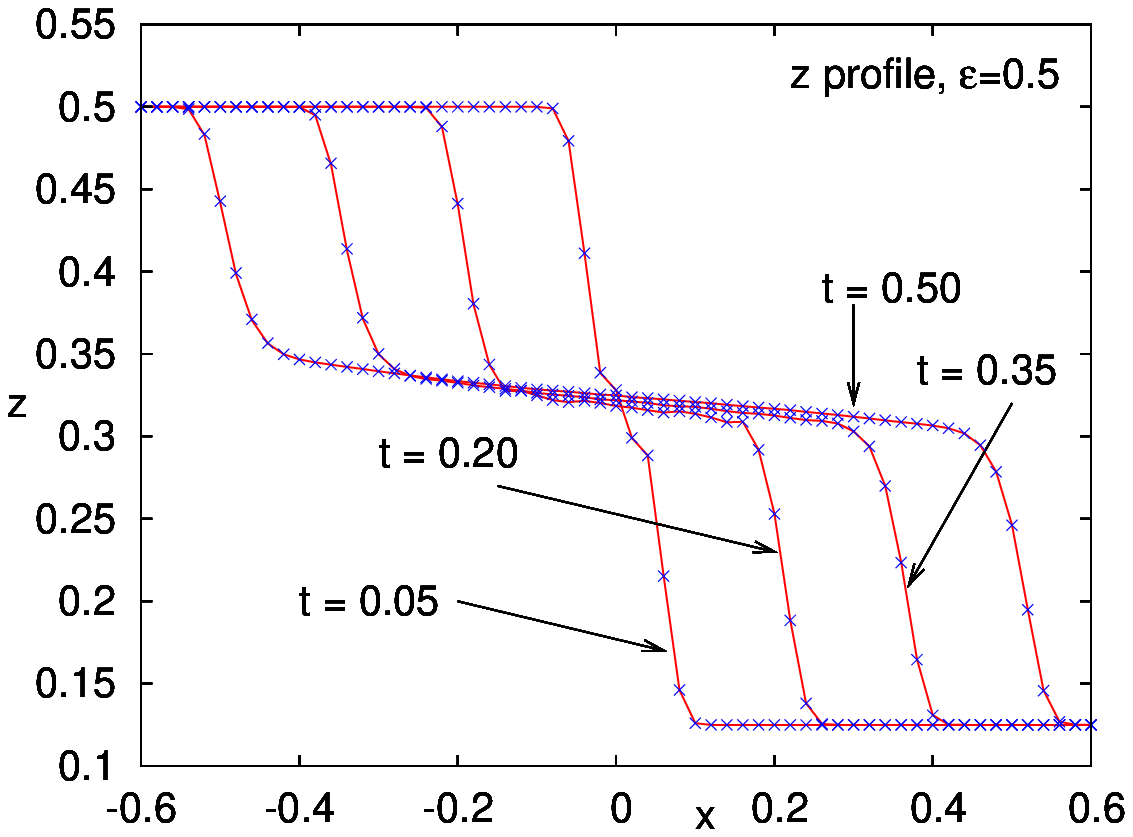}&
\includegraphics[width=7.cm]{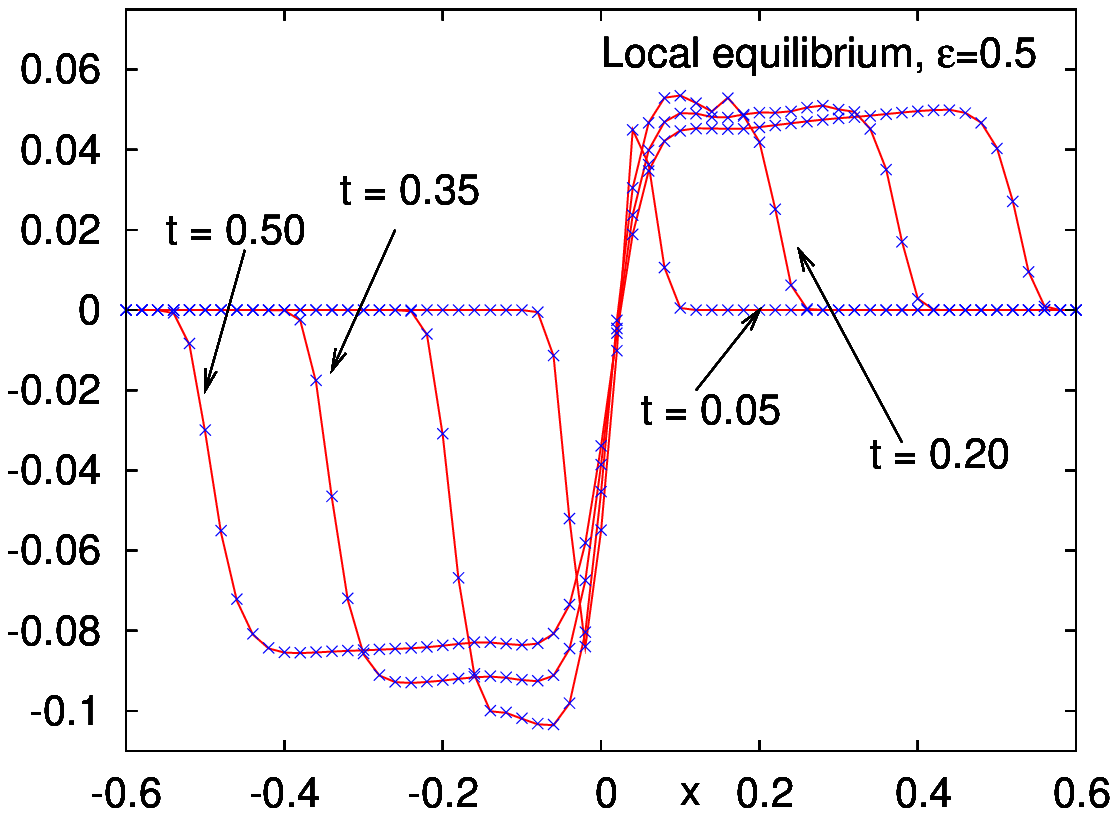}  
\end{tabular}
\caption{ Test 2 : The Riemann problem. {\it Solution to the AP scheme (crosses) and solution the explicit solver (lines) with initial data (\ref{ini1}) in the kinetic regime: $\eps = 0.5$ for the density $\rho$, the momentum $(m,z)$ and the deviation to the local equilibrium $z - A(\rho,m)$.}}
  \label{fig1}
  \end{center}
 \end{figure}

\begin{figure}[ht]
\begin{center}
\begin{tabular}{cc}
\includegraphics[width=7.cm]{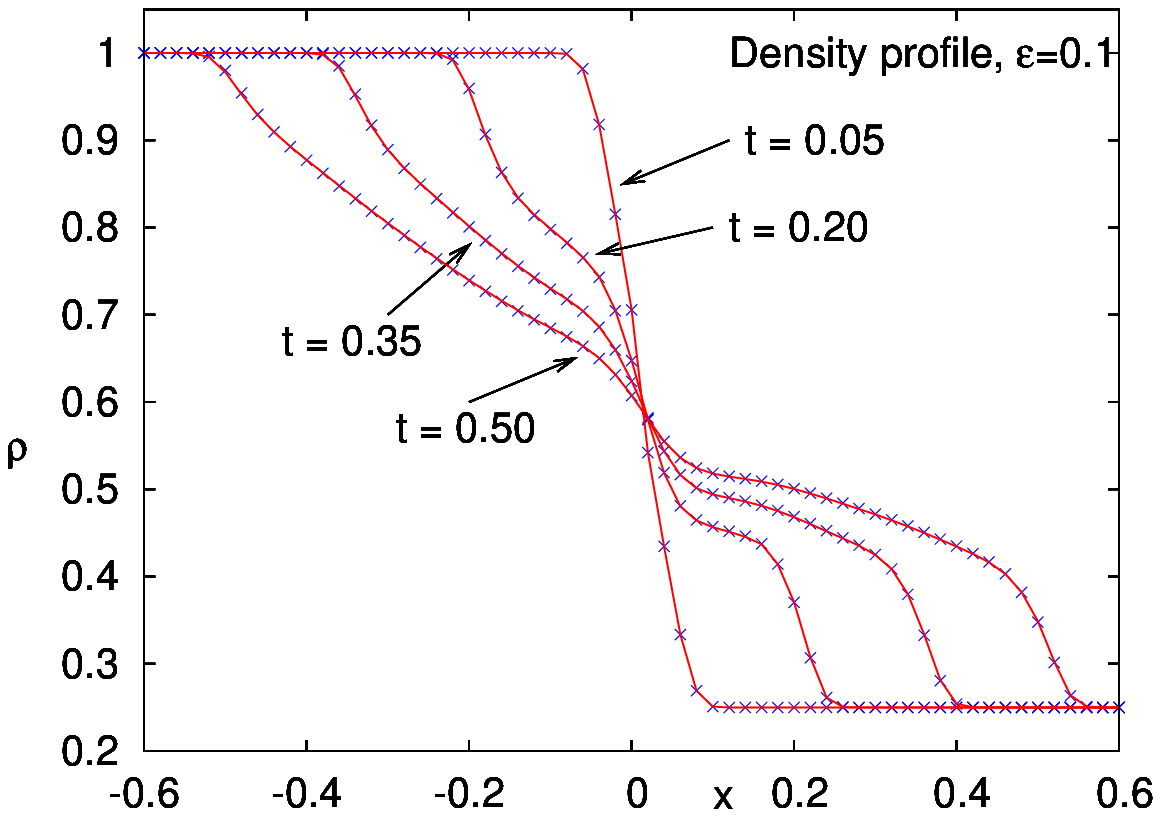}&
\includegraphics[width=7.cm]{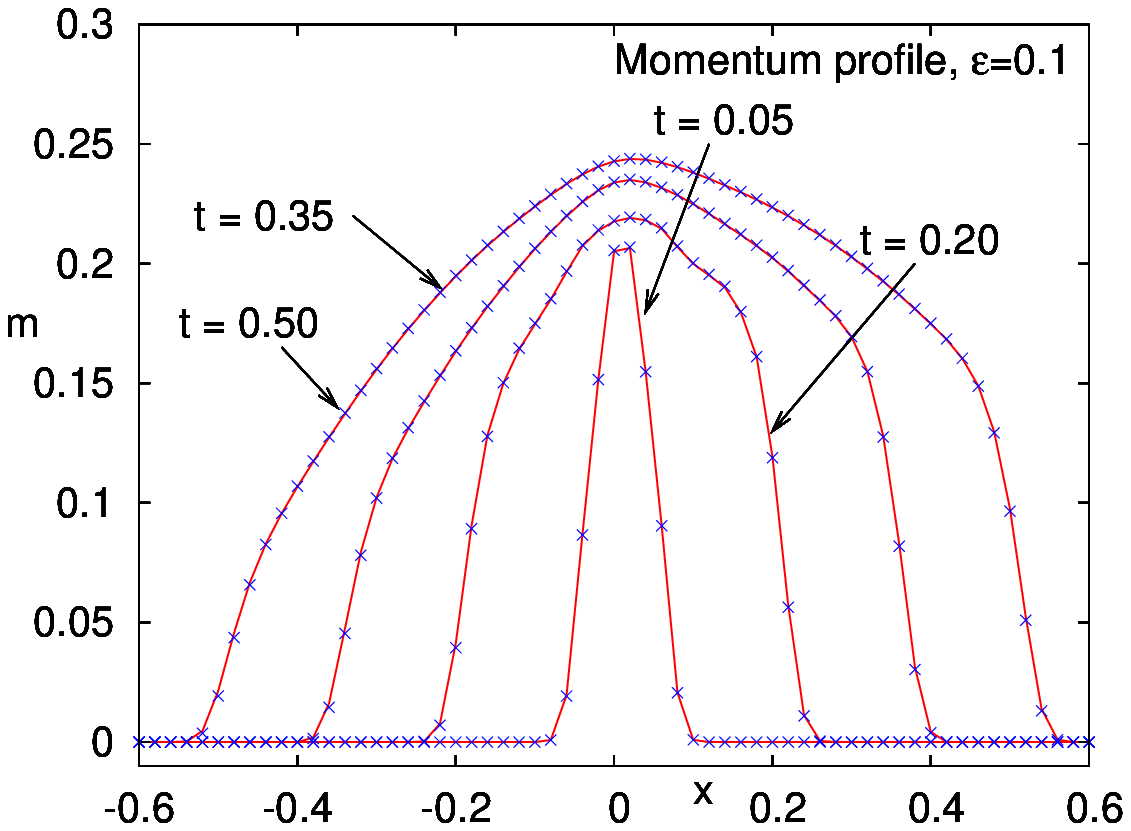} 
\\
\includegraphics[width=7.cm]{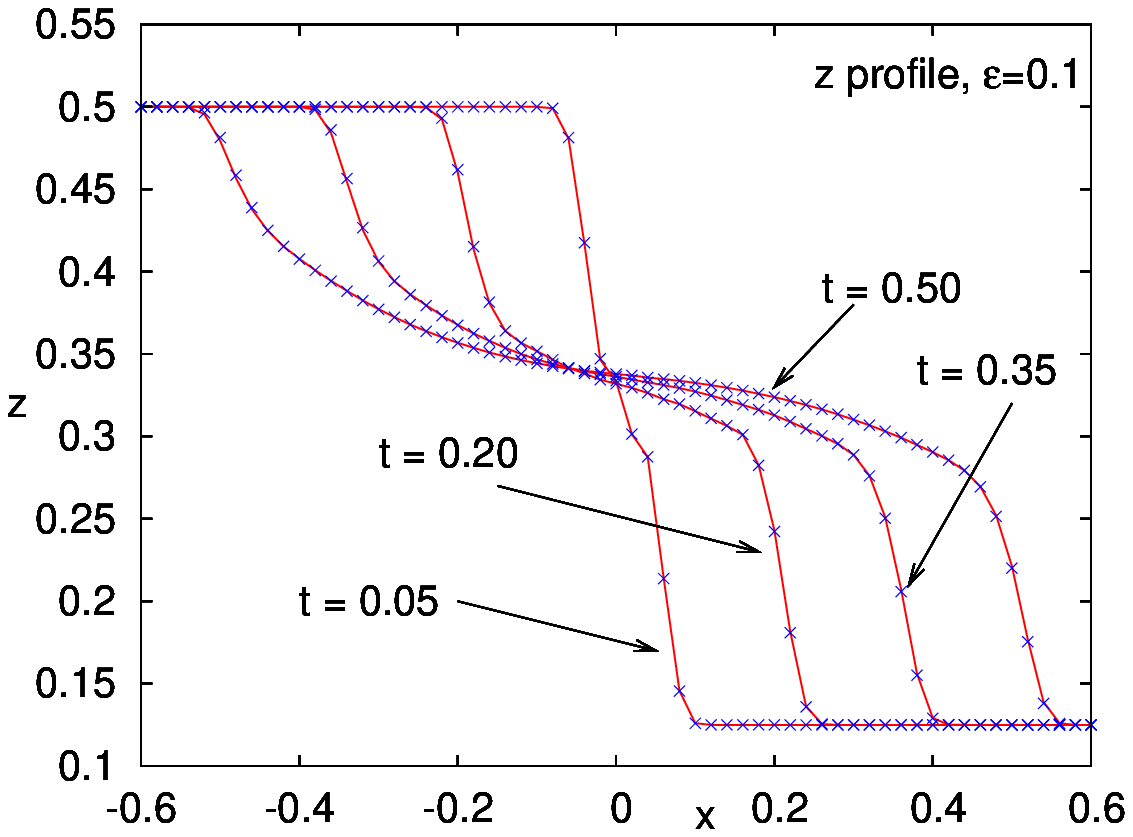}&
\includegraphics[width=7.cm]{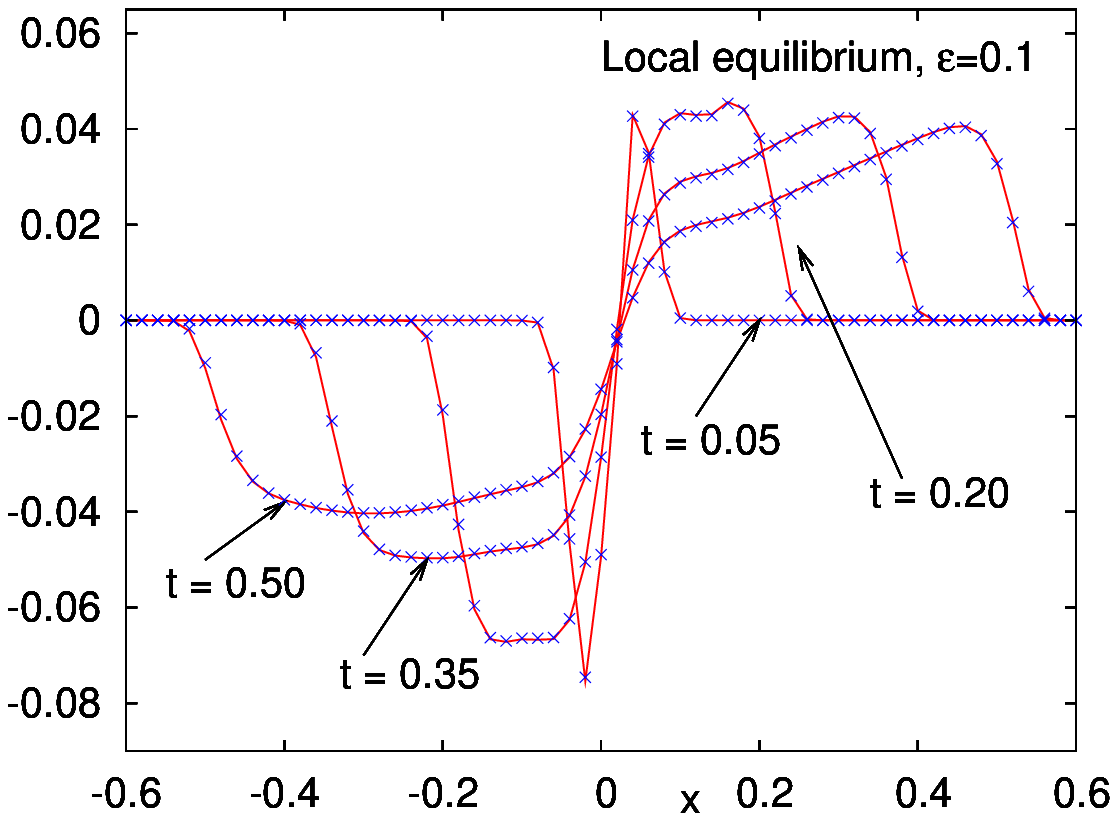}
\end{tabular}
\caption{ Test 2 : The Riemann problem. {\it Solution to the AP scheme (crosses) and solution the explicit solver (lines) with initial data (\ref{ini1}) in the kinetic regime: $\eps = 0.1$ for the density $\rho$, the momentum $(m,z)$ and the deviation to the local equilibrium $z - A(\rho,m)$.}}
  \label{fig1bis}
  \end{center}
 \end{figure}

Next we investigate the cases of small values of $\eps$. The same time step for the AP scheme is used, whereas the fully explicit scheme  requires it to be of order $O(\eps)$. We report the numerical results for $\eps = 10^{-2}$ and $\eps = 10^{-3}$ in Figure \ref{fig1bis}. In this case, we add in the comparison the numerical solution to the hydrodynamic limit  problem (\ref{euler}), obtained with a standard first order finite volume scheme.

\begin{figure}[ht]
\begin{center}
\begin{tabular}{cc}
\includegraphics[width=7.cm]{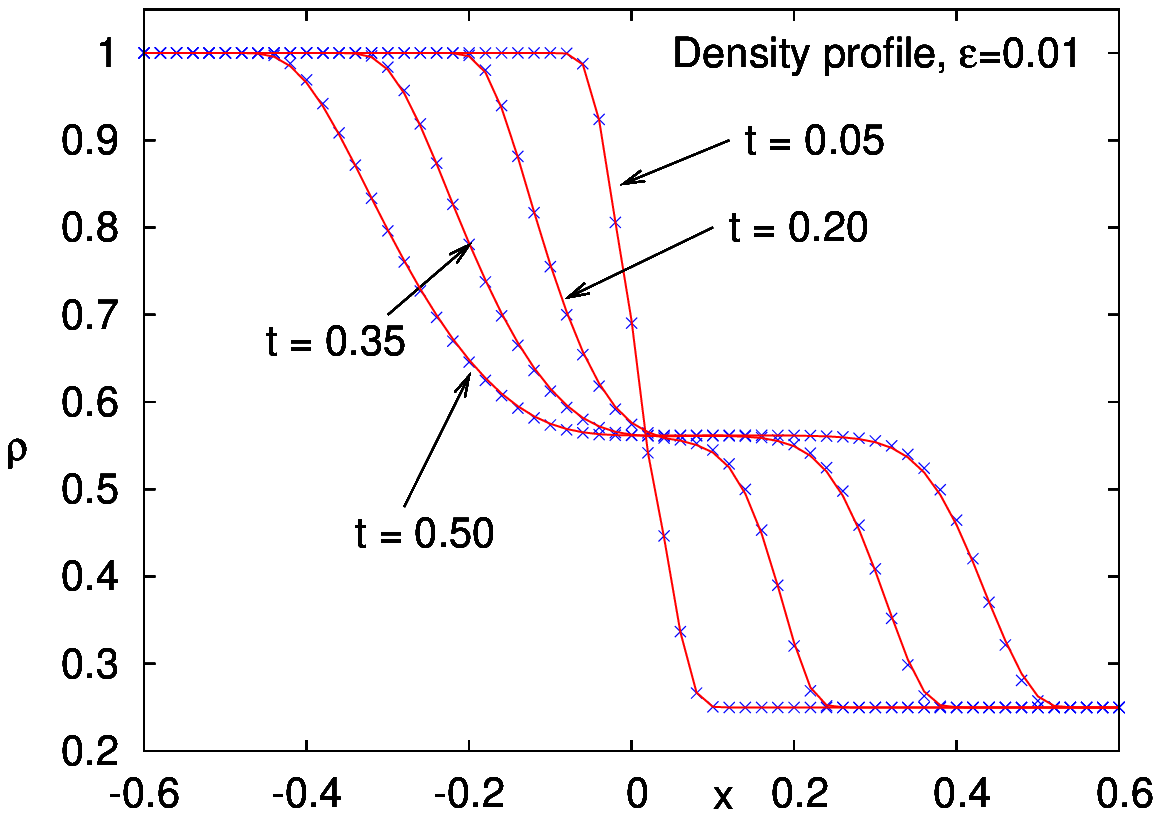}&
\includegraphics[width=7.cm]{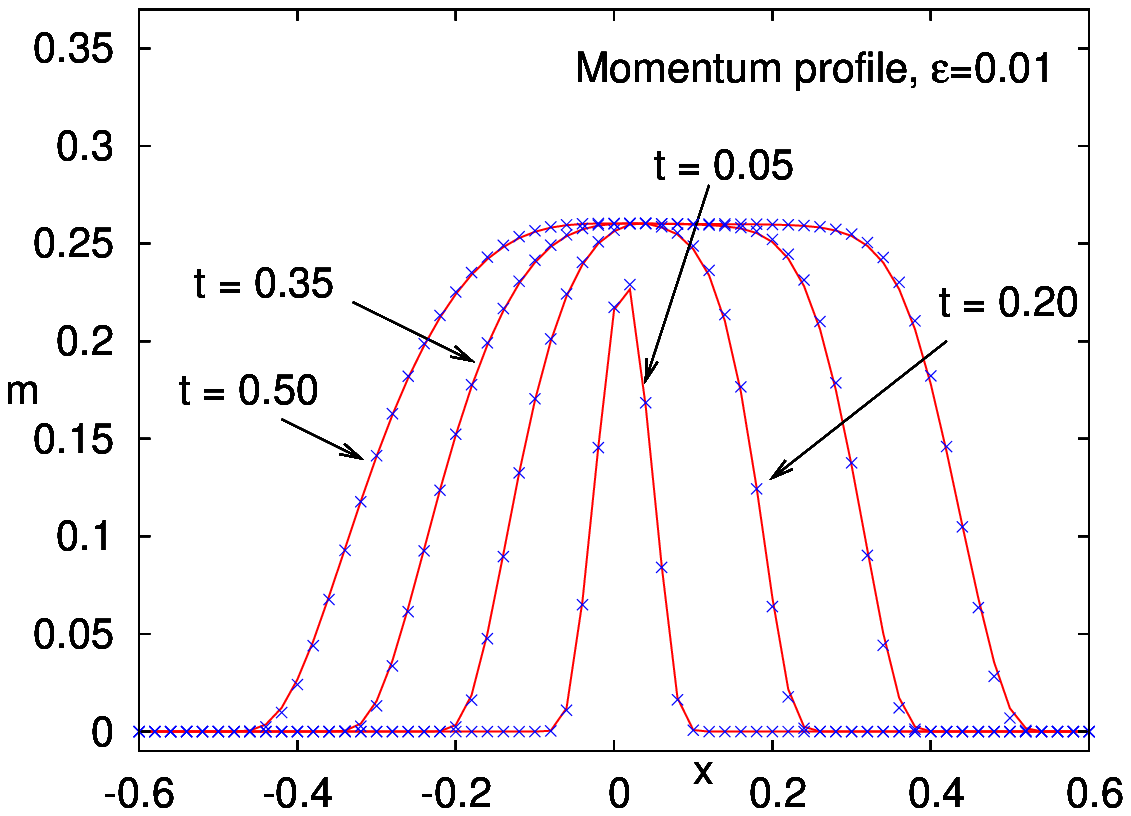}
\\ 
\includegraphics[width=7.cm]{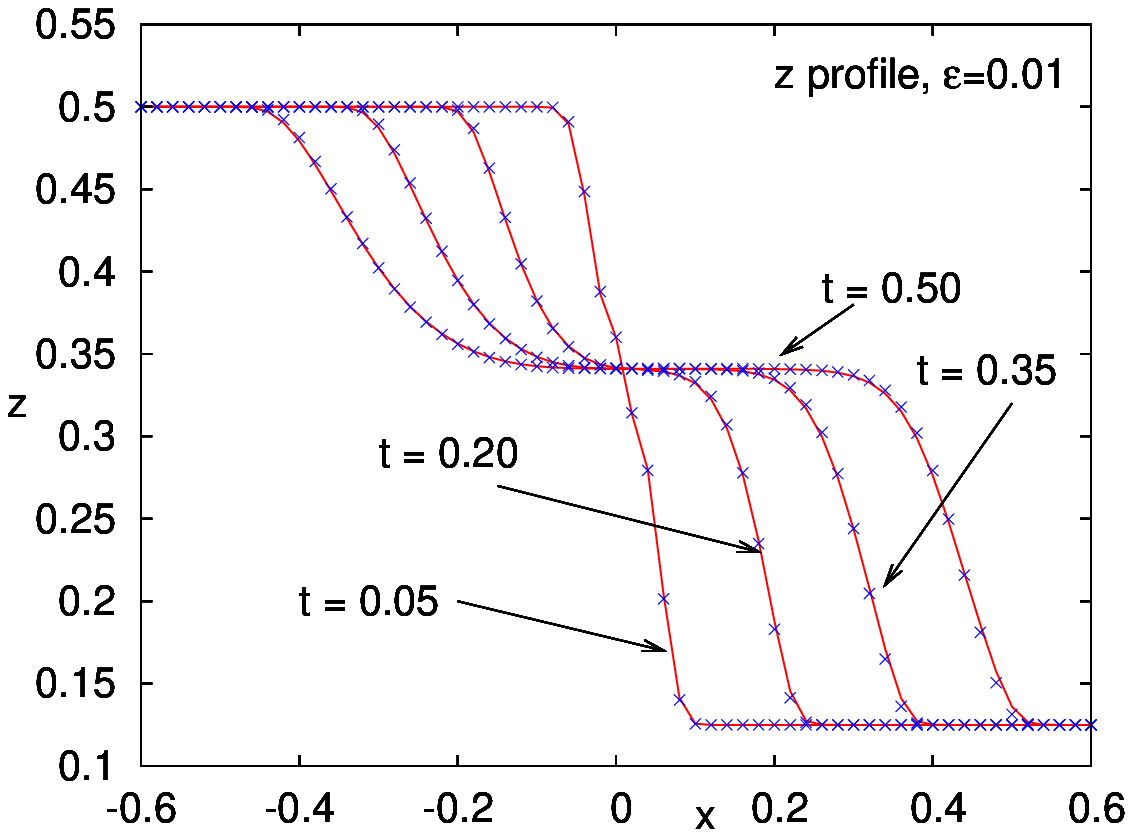}&
\includegraphics[width=7.cm]{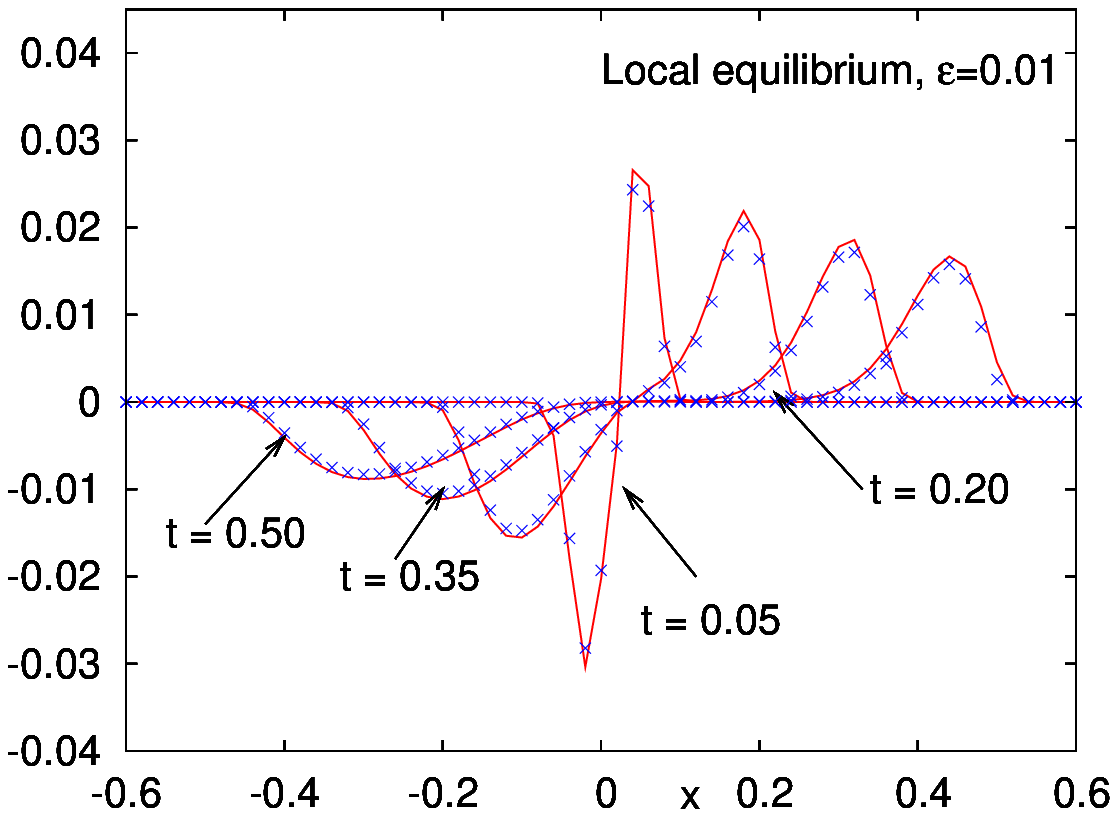}
\\
&\,
\\
&\,
\end{tabular}
\caption{ Test 2 : The Riemann problem.  {\it Solution to the AP scheme (crosses) and solution the explicit solver (lines) with initial data (\ref{ini1}) in the kinetic regime: $\eps = 0.01$ for the density $\rho$, the momentum $(m,z)$ and the deviation to the local equilibrium $z - A(\rho,m)$.}}
\label{fig1ter}
\end{center}
\end{figure}

\begin{figure}[ht]
\begin{center}
\begin{tabular}{cc}
\includegraphics[width=7.cm]{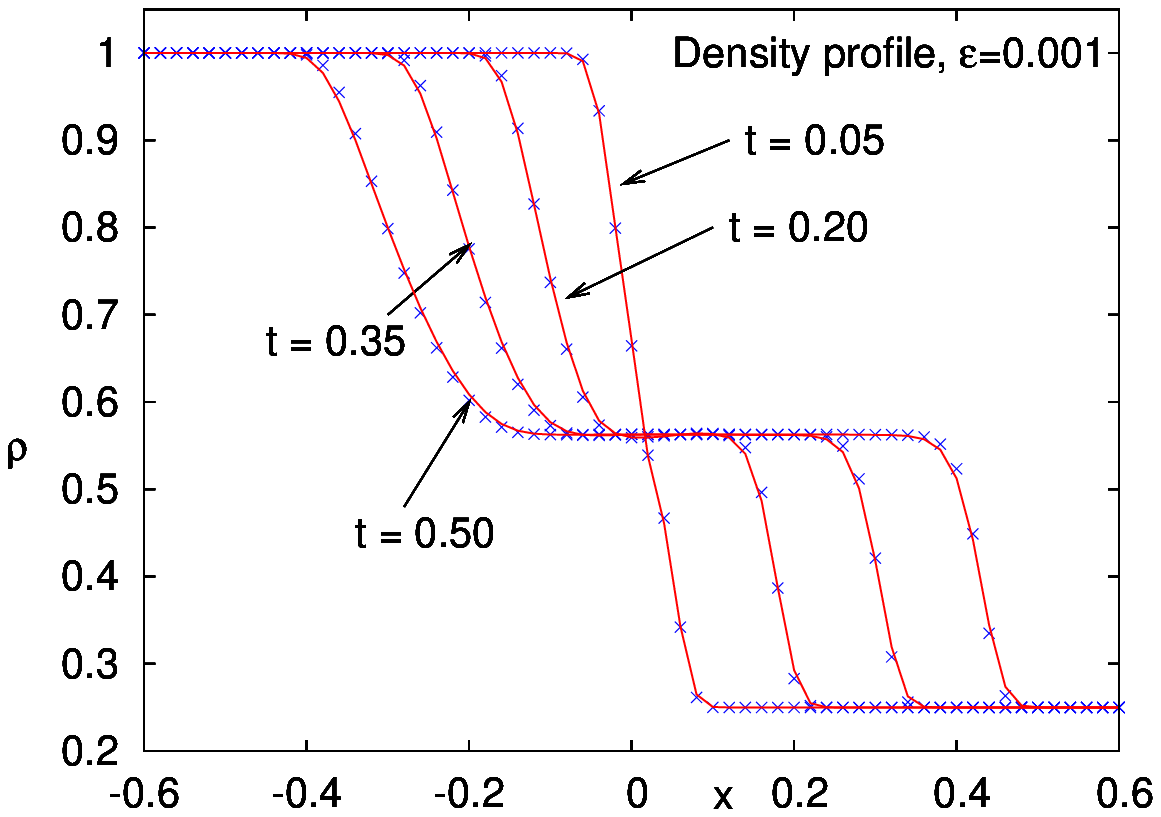}&
\includegraphics[width=7.cm]{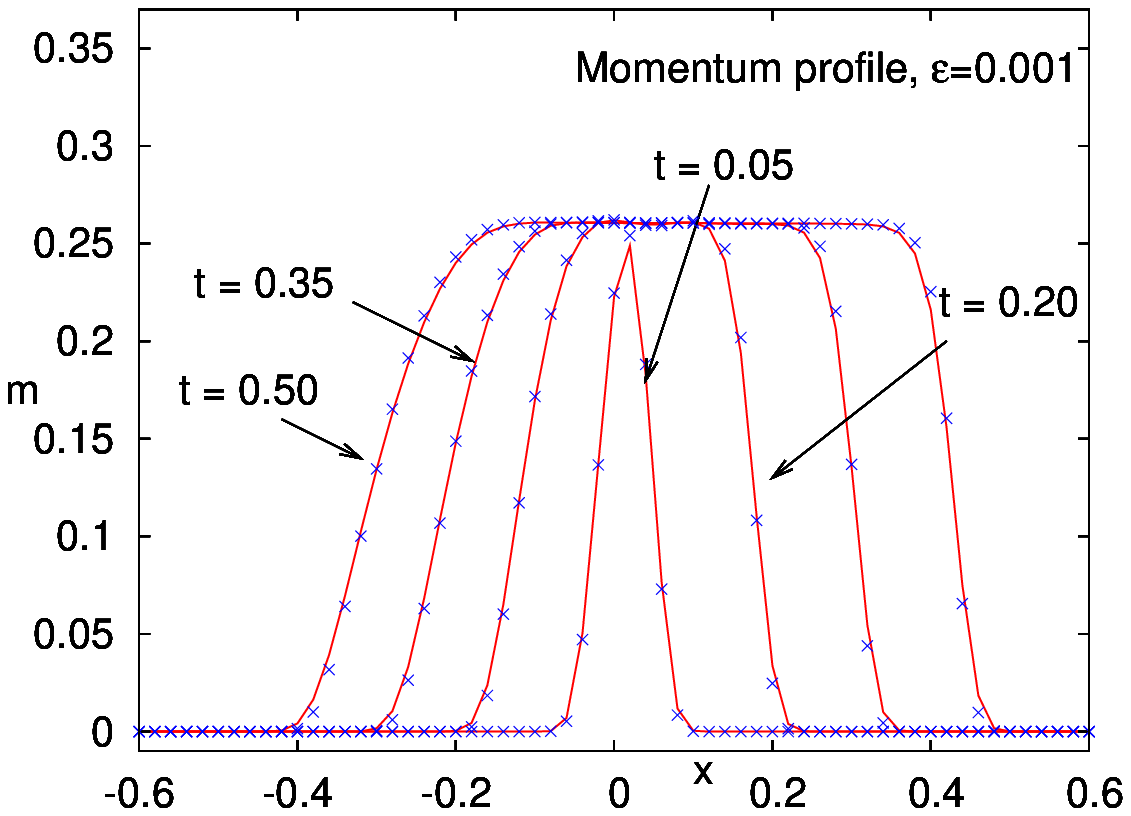} 
\\
\includegraphics[width=7.cm]{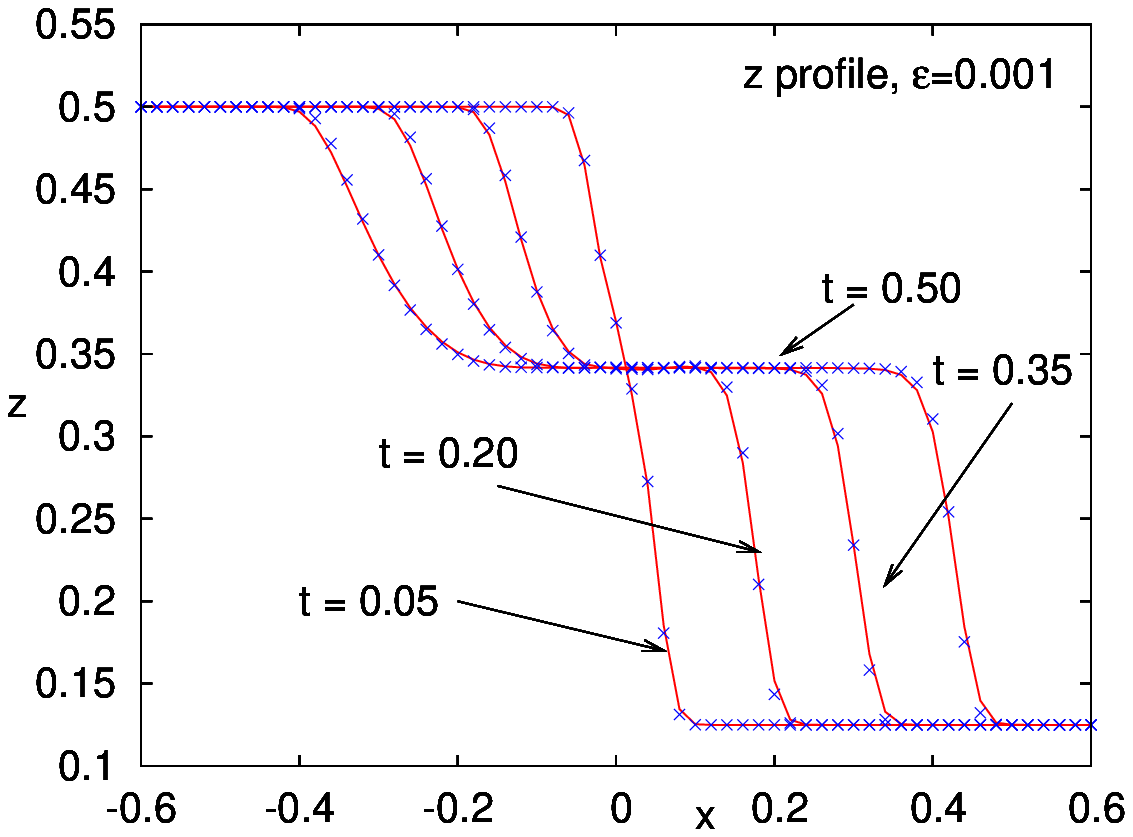}&
\includegraphics[width=7.cm]{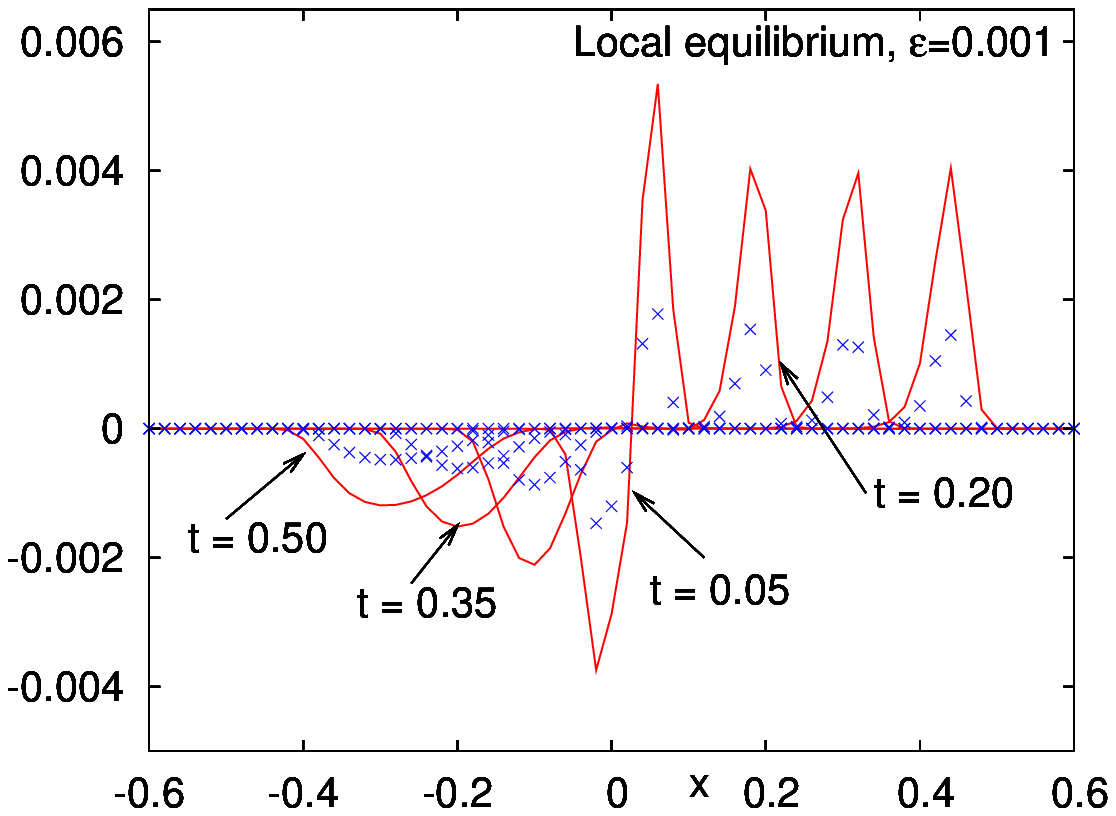}  
\end{tabular}
\caption{ Test 2 : The Riemann problem.  {\it Solution to the AP scheme (crosses) and solution the explicit solver (lines) with initial data (\ref{ini1}) in the kinetic regime: $\eps = 0.01$ for the density $\rho$, the momentum $(m,z)$ and the deviation to the local equilibrium $z - A(\rho,m)$.}}
\label{fig1quat}
\end{center}
\end{figure}

We observe that the AP scheme and the fully explicit scheme still agree, even if the time step is at least ten times larger with our method. Moreover, now that we are closer to the fluid regime, we see that the macroscopic quantities are in good agreement with the ones obtained with the limit problem. Yet some differences between the AP and the explicit schemes can be observed for $\eps = 10^{-3}$, this comes from the fact that we used a very small number of points for both discretizations.

\section{Conclusion}

In this paper we proposed a rigorous convergence proof of an asymptotic preserving  numerical scheme  applied to a system of transport equations with a nonlinear  and stiff source term, for which the asymptotic limit is given by a conservation laws. We have proved the convergence of the approximate solution $(\ueps_h,\veps_h)$ to a nonlinear relaxation system, where  $\eps>0$ is a physical parameter and $h$ represents the discretization parameter. Uniform convergence with respect to $\eps$ and $h$ is proved and error estimates are also obtained. This theoretical result allows justify the efficiency of this new scheme for multi-scale problems.


\begin{flushleft} 
\signff 
\end{flushleft}
\vspace{-4.25cm}
\begin{flushright} 
\signar
\end{flushright}


\begin{thebibliography}{99}

\bibitem{ar_nat}
D. Aregba-Driollet and R. Natalini,
\newblock Convergence of relaxation schemes for conservation laws,
\newblock {\em Appl. Anal.}  {\bf 1-2}  (1996), pp. 163--193.

\bibitem{asympKin2}
C. Bardos, F. Golse and D.  Levermore, 
\newblock Fluid dynamic limits of kinetic equations. I. Formal derivations,
\newblock {\em J. Statist. Phys.}  {\bf 63}  (1991), pp. 323--344.

\bibitem{asympKin3}
C. Bardos, F. Golse and D.  Levermore, 
\newblock Fluid dynamic limits of kinetic equations. II. Convergence proofs for the Boltzmann equation,
\newblock {\em Comm. Pure Appl. Math.}  {\bf 46}  (1993), pp. 667--753.


\bibitem{asympKin4}
C. Bardos, F. Golse and D.  Levermore, 
\newblock Macroscopic limits of kinetic equations,
\newblock {\em IMA Vol. Math. Appl.}  {\bf 29}  (1991), pp. 1--12.



\bibitem{BLM}
M. Bennoune; M. Lemou and L. Mieussens,
\newblock Uniformly stable numerical schemes for the Boltzmann equation preserving the compressible Navier–Stokes asymptotics,
\newblock {\em J. Comput. Phys.}  {\bf 227}  (2008), pp. 3781--3803.


\bibitem{BianHanNat} 
S. Bianchini, B. Hanouzet and R. Natalini,
\newblock  Asymptotic behavior of smooth solutions for partially dissipative hyperbolic systems with a convex entropy.
\newblock {\em Comm. Pure Appl. Math.} {\bf 60} (2007), pp. 1559-1622.


\bibitem{Bia} 
S. Bianchini, 
\newblock  Hyperbolic limit of the Jin-Xin relaxation model,
\newblock {\em Comm. Pure Appl. Math.} {\bf 47} (1994), pp. 787-830.

\bibitem {asympKin1} 
F. Bouchut, F. Golse and M. Pulvirenti, 
\newblock  Kinetic equations and asymptotic theory, 
\newblock {\em Gauthiers-Villars} (2000).


\bibitem{broadwell}
J. E. Broadwell,
\newblock Study of rarefied shear flow by the discrete velocity method.
\newblock {\em J. Fluid Mech.} (1964).

\bibitem{CJR}
R. Caflish; S. Jin  and G. Russo,
\newblock Uniformly accurate schemes for hyperbolic systems with relaxation, 
\newblock {\em SIAM J. Numer. Anal.} {\bf 34} (1997) 246–281. 

\bibitem{chalabi1}
A. Chalabi,
\newblock On convergence of numerical schemes for hyperbolic conservation laws with stiff source terms
\newblock {\em Mathematics of Computation}, {\bf 66}, (1997) {pp. 527--545}

\bibitem{chalabi3}
A. Chalabi,
\newblock Convergence of relaxation schemes for hyperbolic conservation laws with stiff source terms
\newblock {\em Mathematics of Computation}, {\bf 68}, {(1999)}, {pp. 955--970}

\bibitem{chalabi2}
{A. Chalabi and Y. Qiu},
\newblock Relaxation schemes for hyperbolic conservation laws with stiff source terms: application to reacting {E}uler equations
\newblock {\em Journal of Scientific Computing}, {\bf 15}, {(2000)},  {pp. 395--416}

\bibitem{chen}
G.Q. Chen, T.P. Liu and C.D. Levermore,
Hyperbolic conservation laws with stiff relaxation terms and entropy,
{\em Comm. Pure Appl. Math.} {\bf 47} (1994), no. 6, 787--830.
 
\bibitem{DJL}
P. Degond; S. Jin and J.-G. Liu, 
\newblock  Mach-number uniform asymptotic-preserving gauge schemes for compressible flows.  
\newblock {\em Bull. Inst. Math. Acad. Sin. (N.S.)}  {\bf 2}  (2007),  pp. 851--892.

\bibitem{ajout:1}
P. Degond, J.-G. Liu and M-H Vignal,
\newblock    Analysis of an asymptotic preserving scheme for the Euler-Poisson system in the quasineutral limit 
 \newblock {\em SIAM J. Num. Anal.}, {\bf 46} (2008), 1298--1322.

\bibitem{DimarcoP}
G. Dimarco and L. Pareschi,
\newblock Exponential Runge-Kutta methods for stiff kinetic equations, preprint

\bibitem{DimarcoP2}
G. Dimarco and L. Pareschi,
\newblock Hybrid multiscale methods I. Hyperbolic relaxation problems.
\newblock Comm. Math. Sciences, {\bf 4}, (2006) pp. 155--177. 

\bibitem{FiRu:FBE:03}
F. Filbet and G. Russo, 
\newblock High order numerical methods for the space non-homogeneous Boltzmann equation.
\newblock {\em J. Comput. Phys.} {\bf 186}, (2003) pp. 457--480.

\bibitem{FiPa:03} 
F. Filbet; L. Pareschi and G. Toscani,
\newblock Accurate numerical methods for the collisional motion of (heated) granular flows.
\newblock {\em J. Comput. Phys.} {\bf 202}, (2005) pp. 216--235.

\bibitem{FMP} 
F. Filbet; C. Mouhot and L. Pareschi,
\newblock Solving the Boltzmann equation in $N\log\sb 2N$. 
\newblock {\em SIAM J. Sci. Comput.} {\bf 28}, (2006) pp. 1029--1053 

\bibitem{F:08}
F. Filbet,
\newblock An asymptotically stable scheme for diffusive coagulation-fragmentation models,  
\newblock Comm. Math. Sciences, {\bf 6}, (2008) pp. 257--280. 

\bibitem{filb1}
F. Filbet and S. Jin,
\newblock {A class of asymptotic preserving schemes for kinetic equations and related problems with stiff sources},
\newblock J. Comp. Physics, {\bf 229}, (2010)  

\bibitem{filb2}
F. Filbet and S. Jin,
\newblock {An asymptotic preserving scheme for the ES-BGK model for he Boltzmann equation},
\newblock J. Sci. Comp. {\bf 46}, (2011)  

\bibitem{gabetta} 
E. Gabetta; L. Pareschi, and G. Toscani, 
\newblock Relaxation schemes for nonlinear kinetic equations,
\newblock {\em SIAM J. Numer. Anal.}  {\bf 34}  (1997), 2168--2194 

\bibitem{golse_jin_lev} 
F. Golse, S. Jin and C.D. Levermore,
\newblock  The Convergence of Numerical Transfer Schemes in Diffusive Regimes I: The Discrete-Ordinate Method,  
\newblock  SIAM J. Num. Anal. 36 (1999) pp. 1333-1369

\bibitem{jinAP} 
 S. Jin, 
\newblock Effcient asymptotic-preserving (AP) schemes for some multiscale kinetic equations, 
\newblock {\em SIAM J. Sci. Comput.} {\bf 21} (1999) pp. 441–-454,

\bibitem{jinRK}
S. Jin
\newblock Runge-Kutta Methods for Hyperbolic Conservation Laws with Stiff Relaxation Terms, 
\newblock {\rm J. Computational Physics,} {\bf 122} (1995), 51-67. 

\bibitem{JL-TTSP1} 
S. Jin and C.D. Levermore,
\newblock  The Discrete-Ordinate Method in Diffusive Regime, 
\newblock Transp. Theory Stat. Phys. {\bf 20} (1991), 413-439. 

\bibitem{JL}
S. Jin and C.D. Levermore,
 Numerical schemes for hyperbolic conservation laws with stiff relaxation terms. J. Comput. Phys. {\bf 126} (1996), no. 2, 449--467.

\bibitem{JPT1} S. Jin, L. Pareschi and G. Toscani,
Diffusive Relaxation Schemes for Discrete-Velocity Kinetic
Equations, SIAM J. Num. Anal. 35 (1998) 2405-2439

\bibitem{JPT2} 
S. Jin; L. Pareschi and G. Toscani, 
\newblock Uniformly accurate diffusive relaxation schemes for multiscale transport equations.  
\newblock {\em SIAM J. Numer. Anal.}  {\bf 38}  (2000), 913--936 

\bibitem{levB}
C. D. Levermore and B. A. Wagner,
\newblock Robust fluid dynamical closures of the Broadwell model.
\newblock {\em Phys. Lett. A.} {\bf 174}, {220--228} (1993).

\bibitem{liu} 
{T.P. Liu},
\newblock{Hyperbolic conservation laws with relaxation},
\newblock{\em Comm. Math. Phys.}, {\bf 1}, {pp. 153--175}, {(1987)},

\bibitem{ajout:0}
J.-G. Liu and L. Mieussens,
\newblock Analysis of an asymptotic preserving scheme for linear kinetic equations in the diffusion limit 
\newblock{\em SIAM J. Numer. Anal.} {\bf 48} (2010), 1474-1491

\bibitem{nat_han}
{R. Natalini and B. Hanouzet},
\newblock Weakly coupled systems of quasilinear hyperbolic equations,
\newblock {\em Differential Integral Equations}  {\bf 6}  (1996), 1279--1292.


\bibitem{nat_surv} 
R. Natalini, 
Recent results on hyperbolic relaxation problems. Analysis of systems of conservation laws (Aachen, 1997), 128--198, 
Chapman abd Hall/CRC Monogr. Surv. Pure Appl. Math., 99, Chapman \& Hall/CRC, Boca Raton, FL, 1999.

\bibitem{nat} 
R. Natalini, 
Convergence to equilibrium for the relaxation approximations of conservation laws, 
{\rm Comm. Pure Appl. Math.}, {\bf 8}, {1996}, {pp. 795--823},


\bibitem{Par-Ru}
{L. Pareschi and G. Russo},
\newblock Implicit-explicit Runge-Kutta schemes and applications to hyperbolic systems with relaxation.
\newblock J. Sci. Comp. {\bf 25}, (2005).  

\end{thebibliography}
\end{document}